\documentclass[reqno]{amsart}
\usepackage{mathrsfs}
\usepackage{color}
\usepackage{amsmath}
\usepackage{amsfonts}
\usepackage{amssymb}
\usepackage{graphicx}
\usepackage{hyperref}

 \newtheorem{Theorem}{Theorem}[section]
 \newtheorem{Corollary}[Theorem]{Corollary}
 \newtheorem{Lemma}[Theorem]{Lemma}
 \newtheorem{Proposition}[Theorem]{Proposition}
 
 \newtheorem{Definition}[Theorem]{Definition}
 \newtheorem{Problem}[Theorem]{Problem}
 \newtheorem{Conjecture}[Theorem]{Conjecture}

 \newtheorem{Remark}[Theorem]{Remark}

 \numberwithin{equation}{section}

\begin{document}

\title[Concavity property of minimal $L^2$ integrals \uppercase\expandafter{\romannumeral3}]
{Concavity property of minimal $L^2$ integrals with Lebesgue measurable gain \uppercase\expandafter{\romannumeral3}: open Riemann surfaces}

\author{Qi'an Guan}
\address{Qi'an Guan: School of
Mathematical Sciences, Peking University, Beijing 100871, China.}
\email{guanqian@math.pku.edu.cn}

\author{Zheng Yuan}
\address{Zheng Yuan: State Key Laboratory of Mathematical Sciences, Academy of Mathematics and Systems Science, Chinese Academy of Sciences, Beijing 100190, China.}
\email{yuanzheng@amss.ac.cn}

\thanks{}

\subjclass[2020]{32D15, 30F30, 32L10, 32U05, 32W05}

\keywords{optimal $L^{2}$ extension, concavity, Suita conjecture, minimal $L^2$ integral, multiplier ideal sheaf, plurisubharmonic function}

\date{\today}

\dedicatory{}

\commby{}


\begin{abstract}
In this article, we present a characterization of the concavity property of minimal $L^2$ integrals degenerating to linearity in the case of finite points on open Riemann surfaces.
As an application, we give a characterization of the holding of equality in optimal jets $L^2$ extension problem from analytic subsets to open Riemann surfaces, which is a weighted jets version of Suita conjecture for analytic subsets.
\end{abstract}

\maketitle

\tableofcontents

\section{Introduction}\label{introduction}

The strong openness property of multiplier ideal sheaves (i.e. $\mathcal{I}(\varphi)=\mathcal{I}_+(\varphi):=\mathop{\cup} \limits_{\epsilon>0}\mathcal{I}((1+\epsilon)\varphi)$) is an important feature of multiplier ideal sheaves
 and used in the study of several complex variables, complex algebraic geometry and complex differential geometry
(see e.g. \cite{GZSOC,K16,cao17,cdM17,FoW18,DEL18,ZZ2018,GZ20,berndtsson20,ZZ2019,ZhouZhu20siu's,FoW20,KS20,DEL21}),
which was conjectured by Demailly \cite{DemaillySoc} and proved by Guan-Zhou \cite{GZSOC} (the 2-dimensional case was proved by Jonsson-Musta\c{t}$\breve{a}$
\cite{JonssonMustata}),
where $\varphi$ is a plurisubharmonic function of a complex manifold $M$ (see \cite{Demaillybook}), and multiplier ideal sheaf $\mathcal{I}(\varphi)$ is the sheaf of germs of holomorphic functions $f$ such that $|f|^2e^{-\varphi}$ is locally integrable (see e.g. \cite{Tian,Nadel,Siu96,DEL,DK01,DemaillySoc,DP03,Lazarsfeld,Siu05,Siu09,DemaillyAG,Guenancia}).

When $\mathcal{I}(\varphi)=\mathcal{O}$, the strong openness property degenerates to the openness property, which was a conjectured by Demailly-Koll\'ar \cite{DK01}
and proved by Berndtsson \cite{Berndtsson2} (the 2-dimensional case was proved by Favre-Jonsson in \cite{FavreJonsson}).
Recall that Berndtsson \cite{Berndtsson2} established an effectiveness result of the openness property.
Stimulated by Berndtsson's effectiveness result, and continuing the solution of the strong openness property \cite{GZSOC},
Guan-Zhou \cite{GZeff} established an effectiveness result of the strong openness property by considering the minimal $L^{2}$ integrals  on the pseudoconvex domain $D$ related to the multiplier ideal sheaves.

Considering the minimal $L^{2}$ integrals on all sublevels of the weight $\varphi$,
Guan \cite{G16} established a concavity property of the minimal $L^2$ integrals, and used the concavity property to obtain a sharp version of Guan-Zhou's effectiveness result.

\subsection{Concavity property of minimal $L^2$ integrals and optimal $L^2$ extension}
\

Let $D\subset\mathbb{C}^n$ a pseudoconvex domain. Denote the set of all plurisubharmonic functions by $PSH(D)$, and denote $PSH^-(D):=\{\varphi\in PSH(D):\varphi<0\}$. Let $\varphi\in PSH^-(D)$. Let $f$ be a holomorphic function near $z_0\in D.$ For any $t\ge0$, the \textbf{minimal $L^2$ integrals} (see \cite{G16,GZeff}) is defined by  
$$G(t):=\inf\bigg\{\int_{\{\varphi<-t\}}|\tilde f|^2:\tilde f\in\mathcal{O}(\{\varphi<-t\})\,\&\,(\tilde f-f,z_0)\in\mathcal{I}(\varphi)_{z_0}\bigg\}.$$
In \cite{G16}, Guan proved the following concavity property.
\begin{Theorem}
	[\cite{G16}]\label{thm:concavity-guan}
	$G(-\log r)$ is a concave function on $(0,1)$.
\end{Theorem}

As applications of Theorem \ref{thm:concavity-guan}, Guan gave a proof of Saitoh's conjecture for conjugate Hardy $H^2$ kernels \cite{Guan2019},
and presented a sufficient and necessary condition of the existence of decreasing equisingular approximations with analytic singularities for the multiplier ideal sheaves with weights $\log(|z_{1}|^{a_{1}}+\cdots+|z_{n}|^{a_{n}})$ \cite{Guan2020}.

In \cite{G2018} (see also \cite{GM}), Guan gave the concavity property for smooth gain on Stein manifolds.
In \cite{GM_Sci}, Guan-Mi obtained the concavity property for smooth gain on weakly pseudoconvex K\"{a}hler manifolds,
which proved a sharp version of Guan-Zhou's effectiveness result on weakly pseudoconvex K\"{a}hler manifolds.
As applications of the concavity property in \cite{G2018},
Guan-Yuan presented an optimal support function related to the strong openness property \cite{GY-support} and an effectiveness result of the strong openness property in $L^p$ \cite{GY-lp-effe}.
In \cite{GY-concavity}, Guan-Yuan obtained the concavity property with Lebesgue measurable gain on Stein manifolds,
as an application, we presented a twisted version of the strong openness property in $L^p$ \cite{GY-twisted} which gave an affirmative answer to a question posed by Chen in \cite{chen18}.

Recently, Guan-Mi-Yuan \cite{GMY-concavity2} obtained the concavity property with Lebesgue measurable gain on weakly pseudoconvex K\"{a}hler manifolds.

Let $M$ be an $n$-dimensional complex manifold,  $X$ be a closed subset of $M$, and  $Z$ be an analytic subset of $M$. Assume that:

$(1)$  $M\backslash (X\cup Z)$ is a weakly pseudoconvex K\"ahler manifold;

$(2)$ $X$ is locally negligible with respect to $L^2$ holomorphic functions, i.e., for any open subset $U\subset M$ and for any $L^2$ holomorphic function $f$ on $U\backslash X$, there exists an $L^2$ holomorphic function $\tilde{f}$ on $U$ such that $\tilde{f}|_{U\backslash X}=f$ with the same $L^2$ norm.

Let $\psi$ and $\varphi+\psi$ be  plurisubharmonic functions on $M$. Denote $T=-\sup\limits_M \psi$.
\begin{Definition}
	We say that a positive measurable function $c$ (so-called ``gain") on $(T,+\infty)$ in class $\mathcal{P}_{T,M}$ if the following two statements hold:
	\par
	$(1)$ $c(t)e^{-t}$ is decreasing with respect to $t$;
	\par
	$(2)$ there is a closed subset $E$ of $M$ such that $E\subset Z\cap \{\psi(z)=-\infty\}$ and for any compact subset $K\subset M\backslash E$, $e^{-\varphi}c(-\psi)$ has a positive lower bound on $K$.
\end{Definition}

Let $Z_0$ be a subset of $\{\psi=-\infty\}$ such that $Z_0 \cap
Supp(\mathcal{O}/\mathcal{I}(\varphi+\psi))\neq \emptyset$. Let $U \supset Z_0$ be
an open subset of $M$, and let $f$ be a holomorphic $(n,0)$ form on $U$. Let $\mathcal{F}_{z_0} \supset \mathcal{I}(\varphi+\psi)_{z_0}$ be an ideal of $\mathcal{O}_{z_0}$ for any $z_{0}\in Z_0$.

For any $t\ge T$, denote the \textbf{minimal $L^2$ integrals}
\begin{equation}\nonumber
	\begin{split}
		\inf\bigg\{\int_{ \{ \psi<-t\}}|\tilde{f}|^2e^{-\varphi}c(-\psi):& \tilde{f}\in
		H^0(\{\psi<-t\},\mathcal{O} (K_M)  ) \\
		&\&\, (\tilde{f}-f)\in
		H^0(Z_0 ,(\mathcal{O} (K_M) \otimes \mathcal{F})|_{Z_0})\bigg\}
	\end{split}
\end{equation}
by $G(t)$, where $K_M$ is the canonical holomorphic line bundle on $M$, $c$ is a nonnegative function on $(T,+\infty)$, $|f|^2:=\sqrt{-1}^{n^2}f\wedge \bar{f}$ and $(\tilde{f}-f)\in
H^0(Z_0 ,(\mathcal{O} (K_M) \otimes \mathcal{F})|_{Z_0} )$ means $(\tilde{f}-f,z_0)\in(\mathcal{O}(K_M)\otimes \mathcal{F})_{z_0}$ for all $z_0\in Z_0$.

Assume $c\in \mathcal{P}_{T,M}$ and $\int_{T_1}^{+\infty}c(t)e^{-t}dt<+\infty$ for some $T_1>T$. Denote $h(t):=\int_{t}^{+\infty}c(t_1)e^{-t_1}dt_1$. Let us recall the following concavity property of $G(h^{-1}(r))$.
\begin{Theorem}[\cite{GMY-concavity2}, see also \cite{GMY-boundary2}]
	 If  $G(t)\not\equiv+\infty$, then $G(h^{-1}(r))$ is concave with respect to  $r\in (0,\int_{T}^{+\infty}c(t)e^{-t}dt)$, $\lim\limits_{t\to T+0}G(t)=G(T)$ and $\lim\limits_{t \to +\infty}G(t)=0$.
	\label{maintheorem}
\end{Theorem}

The settings of $\varphi$, $\psi$ and $c$ follow from the (optimal) $L^2$ extension theorems. Ohsawa in \cite{Ohsawa3} gave an $L^2$ extension theorem with negligible weights from hyperplanes to bounded pseudoconvex domains in $\mathbb{C}^n$, in which the two plurisubharmonic functions $\varphi$ (denoted by $v$ in \cite{Ohsawa3}) and $\psi+2\log d(\cdot,H)$  first appeared, where $d(\cdot,H)$ is the distance function from the hyperplane $H$.  In \cite{guan-zhou13ap}, Guan-Zhou established an optimal $L^2$ extension theorem in a general setting, in which $\varphi$ and $\psi$ (denoted by $\Psi$ in \cite{guan-zhou13ap}) may not be plurisubharmonic functions and a general class of gain functions $c(t)$ was considered.

Note that a linear function is a degenerate case of a concave function.
It is natural to ask:

\begin{Problem}\label{Q:chara}
How can one characterize the concavity property degenerating to linearity?
\end{Problem}

Some necessary conditions for the concavity property of the minimal $L^2$ integrals degenerating to linearity can be found in \cite{GM,GY-concavity,GMY-concavity2,x-z}.

When $M=\Omega$ is an open Riemann surface and $Z_0$ is a single point set, Guan-Mi \cite{GM} gave an answer to Problem \ref{Q:chara} for  the case where $\varphi$ is subharmonic and $c$ is smooth,
and Guan-Yuan \cite{GY-concavity} gave an answer to Problem \ref{Q:chara} for  the case where  $\varphi$ may not be subharmonic and $c$ is Lebesgue measurable.  

In this article, we consider the case where $M=\Omega$ is an open Riemann surface and $Z_0$ may not be a single point set. We give an answer to Problem \ref{Q:chara} when $Z_0$ is a finite point set (Theorem \ref{thm:m-points}), and we give a necessary condition for the concavity property degenerating to linearity when $Z_0$ is an arbitrary analytic subset of $\Omega$ (Proposition \ref{p:infinite}). The proof is independent of the results for single point case.

\

Another motivation for studying the linear case comes from the (optimal) $L^2$ extension problem for holomorphic sections.

Let us recall the \textbf{(optimal) $L^2$ extension problem} (see \cite{DemaillyAG}, see also \cite{guan-zhou13ap}): 

\emph{Let $Y$ be a complex subvariety of a complex manifold $M$; given a holomorphic function $f$ (or a holomorphic section of a holomorphic vector) on $Y$ satisfying suitable $L^2$ conditions on $Y$, find a holomorphic extension $F$ of $f$ from $Y$ to $M$, together with a good (or even optimal) $L^2$ estimate of $F$ on $M$.
Furthermore, let $F_{\min}$ be the minimal holomorphic extension (which is the holomorphic extension with the minimal $L^2$ integral among all possible extensions), how to character the equality that the $L^2$ integral of $F_{\min}$ equals to the optimal estimate?}

The famous Ohsawa-Takegoshi $L^2$ extension theorem \cite{OT87} solved the existence part of $L^2$ extension problem.

 \begin{Theorem}[\cite{OT87}]
 	Let $D\subset\mathbb{C}^n$ be a bounded pseudoconvex domain, and let $H\subset$ be a complex hyperplane. Let $\varphi\in PSH(D)$. For any $f\in\mathcal{O}(D\cap H)$  satisfying
	$\int_{D\cap H}|f|^2e^{-\varphi}<+\infty,$
	there exists an $F\in\mathcal{O}(D)$ such that $F|_{D\cap H}=f$ and 
	$$\int_{D}|F|^2e^{-\varphi}\le C_D\int_{D\cap H}|f|^2e^{-\varphi},$$
	where $C_D$ is a constant depending only on the diameter of $D$.
\end{Theorem}

After the work of Ohsawa and Takegoshi, the $L^2$ extension problem was widely discussed for various cases and these $L^2$ extension theorems have many applications in the study of several complex variables and complex geometry, e.g., \cite{berndtsson1996,berndtsson annals,berndtsson paun,Blocki07,Chen03,Demaillyshm,DemaillyManivel,D2016,DHP,GZSOC,Manivel,MV07,OT87,Ohsawa2,Ohsawa3,Ohsawa4,Ohsawa5,OhsawaObservation,Popovici,siu74,Siu96,Siu98,Siu02}. Some of these works gave explicit good estimates in the proof of $L^2$ extension theorems, which could be regarded as attempts to the optimal constant problem in the $L^2$ extension theorem.

One of the motivations to consider the optimal estimate in $L^2$ extension problem comes from the Suita's conjecture \cite{suita72} on the comparison between the Bergman kernel $B_{\Omega}(z_0)$ (see \cite{Berg70}) and logarithmic capacity $c_{\beta}(z_0)$ (see \cite{S-O69}) on open Riemann surfaces. In \cite{OhsawaObservation}, Ohsawa observed a relation between the $L^2$ extension theorem with the inequality part of Suita's conjecture.  
 
 \begin{Conjecture}
 	[Suita Conjecture \cite{suita72}]Let $\Omega$ be an open Riemann surface, which admits a nontrivial Green function. Then $$(c_{\beta}(z_0))^2\le\pi B_{\Omega}(z_0)$$ 
 	for any $z_0\in\Omega$, and equality holds if and only if $\Omega$  is conformally equivalent to the unit disc less a (possible) closed	set of inner capacity zero.
 \end{Conjecture}
 
  Suita proved that $B_{\Omega}(z)=\frac{1}{\pi}\frac{\partial^2}{\partial z\partial\overline{z}}\log c_{\beta}(z)$, thus there is a geometric interpretation of Suita conjecture (see \cite{suita72}): the curvature of the metric $c_{\beta}^2|dz|^2$ is bounded above by $-4$, and it equals $-4$ if and only if $\Omega$  is conformally equivalent to the unit disc less a (possible) closed set of inner capacity zero.
 
The  Suita conjecture  corresponds to the following optimal $L^2$ extension problem: the case of extending from a single point to open Riemann surfaces with trivial weights.

\emph{\textbf{Inequality part:} There exists a holomorphic $(1,0)$ form $F$ on $\Omega$ such that $F(z_0)=dz$ and $\int_{\Omega}|F|^2\le \frac{2\pi}{(c_{\beta}(z_0))^2}$, where $z$ is a local coordinate near $z_0\in\Omega$;}

\emph{\textbf{Equality part:} For the minimal holomorphic extension form $F_{\min}$, equality $\int_{\Omega}|F_{\min}|^2= \frac{2\pi}{(c_{\beta}(z_0))^2}$ holds if and only if $\Omega$  is conformally equivalent to the unit disc less a (possible) closed	set of inner capacity zero.}

\


In \cite{GZZ11} (see also \cite{ZGZ}), a method of introducing undetermined functions with using ODE was initiated to approach the optimal constant problem in the $L^2$ extension theorem. For bounded pseudoconvex domains in $\mathbb{C}^n$, Blocki \cite{Blocki-inv} developed the equation of undetermined functions in \cite{GZZ11} (see also \cite{ZGZ} and \cite{Blocki12}) and got the Ohsawa-Takegoshi $L^2$ extension theorem with an optimal estimate which deduced the inequality part of the Suita conjecture \cite{suita72} for bounded planar domains. Continuing the previous work \cite{ZGZ}, Guan-Zhou (see \cite{GZsci}, see also \cite{guan-zhou CRMATH2012}) gave an optimal $L^2$ extension theorem with negligible weights on Stein manifolds which deduced the inequality part of the Suita conjecture for open Riemann surfaces.
Subsequently,  Guan-Zhou \cite{guan-zhou13ap} established an optimal $L^2$ extension theorem in a general setting on Stein manifolds, as applications, they proved the equality part of Suita conjecture and  gave a geometric meaning of the optimal $L^2$ extension theorem. After that, some further developments and applications of the optimal $L^2$ extension can be found in \cite{BL16},\cite{cao17},\cite{HPS18},\cite{PT18},\cite{ZZ2018},\cite{ZZ2019},\cite{ZhouZhu20siu's}, and so on. The jets version of the $L^2$ extension, as a generalization of the classical Ohsawa-Takegoshi $L^2$ extension theorem, was considered by Popovici in \cite{Popovici} (see \cite{D2016,RZ2021} for various settings). In \cite{Hoso2020} and \cite{ZhouZhu2022}, the optimal jets $L^2$ extension theorems were been established.

The optimal $L^2$ extension theorems extended the inequality part of Suita's conjecture to general cases (general manifolds, subvarieties and weights). Therefore, a natural question is

\begin{Problem}\label{problem:2}
Can one prove the equality part of Suita conjecture for general cases, i.e.,
characterize the equality in the optimal $L^2$ extension problem in general settings?
\end{Problem}

The equality in optimal $L^2$ extension problem  is that

\centerline{\textbf{Minimal $L^2$ integral of holomorphic extensions $=$ optimal estimate.}}
 If the above equality does not hold, there exists a holomorphic extension $F$ of $f$ from $Y$ to $M$ such that the $L^2$ integral of $F$ is strictly smaller than the optimal estimate.

The solution of the equality part of (extended) Suita conjecture (see \cite{guan-zhou13ap,Yamada}) gives 
a characterization of the holding of equality in optimal $L^2$ extension problem of extending from a single point to open Riemann surfaces with trivial weights (or harmonic weights). 
For the general case, we have the following observation:

\emph{If the equality in optimal $L^2$ extension problem holds, the concavity property for the corresponding minimal $L^2$ integrals of holomorphic extensions on the sublevel sets must degenerate to linearity (see \cite{GY-concavity,GM}).}

Thus, studying the linearity case of minimal $L^2$ integrals will aid in studying the equality in  optimal $L^2$ extension problem. Based on the researches on concavity property of minimal $L^2$ integrals, Guan-Mi \cite{GM} gave a solution to Problem \ref{problem:2} for the case of extending from single point to open Riemann surfaces with subharmonic weights, and Guan-Yuan generalized it to the case of weights that may not be subharmonic. In \cite{GMY-concavity2},  Guan-Mi-Yuan proved a weighted jets version of Suita conjecture. The above-mentioned results concern extensions from a single point to open Riemann surfaces, and their proofs all rely on the solution of the (extended) Suita conjecture.

In this article, utilizing the results on Problem \ref{Q:chara} (Theorem \ref{thm:m-points} and Proposition \ref{p:infinite}), we provide a characterization of the holding of equality in optimal jets $L^2$ extension problem from arbitrary analytic subsets to open Riemann surfaces, which proves the weighted jets version of Suita conjecture for analytic subsets and gives an answer to Problem \ref{problem:2} for the open Riemann surfaces case. It is worth noting that this proof is independent of the solution of the (extended) Suita conjecture. When the equality in optimal jets $L^2$ extension problem holds, we give an expression of the minimal holomorphic extension form $F_{\min}$.

\subsection{A characterization for linearity to hold on open Riemann surfaces}
\label{sec:1.2}
\

Let $\Omega$  be an open Riemann surface, which admits a nontrivial Green function $G_{\Omega}$ ($G_{\Omega}<0$ on $\Omega$), and $K_{\Omega}$ be the canonical (holomorphic) line bundle on $\Omega$. Take 
$Z_0:=\{z_1,z_2,\ldots,z_m\}\subset\Omega$ 
be a subset of $\Omega$ satisfying that $z_j\not=z_k$ for any $j\not=k$.

 Denote the set of subharmonic functions on $\Omega$ by $SH(\Omega)$ ($SH^-(\Omega)$ denotes all negative subharmonic functions on $\Omega$). Let $\psi\in SH^-(\Omega)$, and $\varphi$ be a  function on $\Omega$ such that $\varphi+\psi\in SH(\Omega)$. Let $c(t)$ be a positive function on $(0,+\infty)$ satisfying $c(t)e^{-t}$ is decreasing on $(0,+\infty)$,  $\int_{0}^{+\infty}c(s)e^{-s}ds<+\infty$ and $e^{-\varphi}c(-\psi)$ has  a  positive lower bound on any compact subset of $\Omega\backslash E$, where $E\subset\{\psi=-\infty\}$ is a discrete point subset of $\Omega$. Let $f$ be a holomorphic $(1,0)$ form on a neighborhood of $Z_0$, and  $\mathcal{F}_{z_j}\supset\mathcal{I}(\varphi+\psi)_{z_j}$ be an ideal of $\mathcal{O}_{z_j}$ for any $j\in\{1,2,\ldots,m\}$.
Denote 
\begin{equation*}
\begin{split}
\inf\bigg\{\int_{\{\psi<-t\}}|\tilde{f}|^{2}e^{-\varphi}c(-\psi):(\tilde{f}-f,z_j)\in(\mathcal{O}(K_{\Omega})&)_{z_j}\otimes\mathcal{F}_{z_j} \,for\,j\in\{1,2,\ldots,m\} \\&\&{\,}\tilde{f}\in H^{0}(\{\psi<-t\},\mathcal{O}(K_{\Omega}))\bigg\}
\end{split}
\end{equation*}
by $G(t;c)$ (without misunderstanding, we denote $G(t;c)$ by $G(t)$),  where $t\in[0,+\infty)$.  $G(h^{-1}(r))$ is concave with respect to $r$ (by Theorem \ref{maintheorem}, see also \cite{GY-concavity,GMY-concavity2}), where  $h(t)=\int_{t}^{+\infty}c(s)e^{-s}ds$ for any $t\ge0$.

Before stating the main result,  we recall some notations (see \cite{OF81}, see also \cite{guan-zhou13ap,GY-concavity,GMY-concavity2}).
 Let $z_0\in\Omega$.
 Let $P:\Delta\rightarrow\Omega$ be the universal covering from unit disc $\Delta$ to $\Omega$.
 We call  $f\in\mathcal{O}(\Delta)$  a multiplicative function,
 if there is a character $\chi$, which is the representation of the fundamental group of $\Omega$, such that $g^{*}f=\chi(g)f$,
 where $|\chi|=1$ and $g$ is an element of the fundamental group of $\Omega$. Denote the set of such kinds of $f$  by $\mathcal{O}^{\chi}(\Omega)$.

It is known that for any harmonic function $u$ on $\Omega$,
there exists a $\chi_{u}$ and a multiplicative function $f_u\in\mathcal{O}^{\chi_u}(\Omega)$,
such that $|f_u|=P^{*}\left(e^{u}\right)$.
If $u_1-u_2=\log|f|$, then $\chi_{u_1}=\chi_{u_2}$,
where $u_1$ and $u_2$ are harmonic functions on $\Omega$ and $f\in\mathcal{O}(\Omega)$.
Recall that for the Green function $G_{\Omega}(z,z_0)$,
there exist a $\chi_{z_0}$ and a multiplicative function $f_{z_0}\in\mathcal{O}^{\chi_{z_0}}(\Omega)$, such that $|f_{z_0}(z)|=P^{*}\left(e^{G_{\Omega}(z,z_0)}\right)$ (see \cite{suita72}).

The following theorem gives a characterization of $G(h^{-1}(r))$ being linear, which is an answer to Problem \ref{Q:chara} in the case of finite points on open Riemann surfaces.

\begin{Theorem}
	\label{thm:m-points}
 Assume that $G(0)\in(0,+\infty)$ and $(\psi-2p_jG_{\Omega}(\cdot,z_j))(z_j)>-\infty$ for  $j\in\{1,2,\ldots,m\}$, where $p_j=\frac{1}{2}\nu(dd^c(\psi),z_j)>0$ and $\nu(dd^c(\psi),z_j)$ is the Lelong number of $\psi$ at $z_j$. 
 Then $G(h^{-1}(r))$ is linear with respect to $r\in(0,\int_0^{+\infty}c(t)e^{-t}dt)$ if and only if the following statements hold:
	
	$(1)$ $\psi=2\sum_{1\le j\le m}p_jG_{\Omega}(\cdot,z_j)$;
	
	$(2)$ $\varphi+\psi=2\log|g|+2\sum_{1\le j\le m}G_{\Omega}(\cdot,z_j)+2u$ and $\mathcal{F}_{z_j}=\mathcal{I}(\varphi+\psi)_{z_j}$ for any $j\in\{1,2,\ldots,m\}$, where $g\in\mathcal{O}(\Omega)$ such that $ord_{z_j}(g)=ord_{z_j}(f)$ for any $j\in\{1,2,\ldots,m\}$ and $u$ is a harmonic function on $\Omega$;
	
	$(3)$ $\prod_{1\le j\le m}\chi_{z_j}=\chi_{-u}$, where $\chi_{-u}$ and $\chi_{z_j}$ are the  characters associated to the functions $-u$ and $G_{\Omega}(\cdot,z_j)$ respectively;
	
	$(4)$ There is a constant $c_0\in\mathbb{C}\backslash\{0\}$ s.t. $\lim\limits_{z\rightarrow z_k}\frac{f}{gP_*\Big(f_u\big(\prod\limits_{1\le j\le m}f_{z_j}\big)\big(\sum\limits_{1\le j\le m}p_{j}\frac{d{f_{z_{j}}}}{f_{z_{j}}}\big)\Big)}=c_0$ for any $k\in\{1,2,\ldots,m\}$.
\end{Theorem}

When $m=1$, Theorem \ref{thm:m-points} can be found in \cite{GY-concavity}. For the case that $Z_0$ is an infinite set, we give a necessary condition for $G(h^{-1}(r))$ being linear in Section \ref{sec:4.2}. 

If $\Omega$  does not admit a nontrivial Green function, there is no $\psi\in SH^-(\Omega)$ such that $\nu(dd^c(\psi),z_j)>0$. For this case,  deleting the requirement ``$\psi<0$",  we can also obtain a  characterization of $G(h^{-1}(r))$ being linear. In fact, $\{\psi<-t\}$ is an open Riemann surface, which admits a nontrivial Green function, thus we can use Theorem \ref{thm:m-points} by replacing $\Omega$ by $\{\psi<-t\}$ for all $t\in\mathbb{R}$ to obtain the characterization of $G(h^{-1}(r))$ being linear on $(0,\int_{-\infty}^{+\infty}c(t)e^{-t}dt)$.

\begin{Remark}
	 For any $\{z_1,z_2,\ldots,z_m\}$, there exists a harmonic function $u$ on $\Omega$ such that $\prod_{1\le j\le m}\chi_{z_j}=\chi_{-u}$. In fact, as $\Omega$ is an open Riemann surface, there exists  $\tilde{f}\in\mathcal{O}(\Omega)$ satisfying that $u:=\log|\tilde{f}|-\sum_{1\le j\le m}G_{\Omega}(\cdot,z_j)$ is harmonic on $\Omega$, which implies that $\prod_{1\le j\le m}\chi_{z_j}=\chi_{-u}$.
\end{Remark}

We give an expression of the ``minimal" holomorphic form on $\Omega$ when the linearity holds in Theorem \ref{thm:m-points}.
\begin{Remark}\label{rem:1.1}
When the four statements in Theorem \ref{thm:m-points} hold,
$$F:=c_0gP_*\Bigg(f_u\bigg(\prod\limits_{1\le j\le m}f_{z_j}\bigg)\bigg(\sum\limits_{1\le j\le m}p_{j}\frac{d{f_{z_{j}}}}{f_{z_{j}}}\bigg)\Bigg)$$
is a holomorphic $(1,0)$ form  on $\Omega$, and it
 is the unique ``minimal" holomorphic $(1,0)$ form on all sublevel sets $\{\psi<-t\}$, i.e.,  $(F-f,z_j)\in(\mathcal{O}(K_{\Omega}))_{z_j}\otimes\mathcal{F}_{z_j}$ for any $j\in\{1,2,\ldots,m\}$ and
	$G(t)=\int_{\{\psi<-t\}}|F|^2e^{-\varphi}c(-\psi)$
	for any $t\ge0$.
\end{Remark}

\subsection{Optimal jets $L^2$ extension problem and generalized Suita conjecture}\label{sec:1.3}
\

In this section, as an application of Theorem \ref{thm:m-points}, we obtain a characterization of the holding of equality in optimal jets $L^2$ extension problem from finite points to open Riemann surfaces. The result for infinite points case can be found in Section \ref{sec:5.3}. Based on these two results, we prove a weighted jets version of Suita conjecture for analytic subsets, which gives an answer to Problem \ref{problem:2} for the open Riemann surfaces case.

Let $\Omega$, $Z_0$, $z_j$, $\varphi$ and $\psi$  be as in Section \ref{sec:1.2}. Let $w_j$ be a local coordinate on a neighborhood $V_{z_j}\Subset\Omega$ of $z_j$ satisfying $w_j(z_j)=0$ for $j\in\{1,2,\ldots,m\}$, where $V_{z_j}\cap V_{z_k}=\emptyset$ for any $j\not=k$. Denote that $V_0:=\cup_{1\le j\le n}V_{z_j}$.  
The logarithmic capacity $c_{\beta}(z)$  (see \cite{S-O69}) is locally defined by
$$c_{\beta}(z_j):=\exp\lim_{z\rightarrow z_j}(G_{\Omega}(z,z_j)-\log|w_j(z)|).$$ 

Let $k_j$ be a nonnegative integer for any $j\in\{1,2,\ldots,m\}$. Assume that    $p_j:=\frac{1}{2}v(dd^{c}\psi,z_j)>0$, $\frac{1}{2}\nu(dd^c(\varphi+\psi),z_j)=k_j+1$ and $\alpha_j:=(\varphi+\psi-2(k_j+1)G_{\Omega}(\cdot,z_j))(z_j)>-\infty$ for any $j$. Let $c(t)$ be a positive measurable function on $(0,+\infty)$ satisfying $c(t)e^{-t}$ is decreasing on $(0,+\infty)$ and $\int_{0}^{+\infty}c(s)e^{-s}ds<+\infty$. 

We give a characterization of the holding of equality in optimal jets $L^2$ extension problem from finite points to open Riemann surfaces.
\begin{Theorem}
\label{c:L2-1d-char}Let $f\in H^0(V_0,\mathcal{O}(K_{\Omega}))$ satisfy $f=a_jw_j^{k_j}dw_j$ on $V_{z_j}$ for any $j$,
  $a_j$ is a sequence of constants such that $\sum_{1\le j\le m}|a_j|\not=0$. Then there exists  $F\in H^0(\Omega,\mathcal{O}(K_{\Omega}))$  such that $F=f+o(w_j^{k_j})dw_j$ near $z_j$ for any $j$ and
 \begin{equation}
 	\label{eq:210902a}
 	\int_{\Omega}|F|^2e^{-\varphi}c(-\psi)\leq\left(\int_0^{+\infty}c(s)e^{-s}ds\right)\sum_{1\le j\le m}\frac{2\pi|a_j|^2e^{-\alpha_j}}{p_jc_{\beta}(z_j)^{2(k_j+1)}}.
 \end{equation}

 Moreover,
 denoting that the minimal $L^2$ integral of holomorphic extensions
 $C_{\Omega,f}:=\inf\big\{\int_{\Omega}|\tilde{F}|^2e^{-\varphi}c(-\psi):\tilde{F}\in H^0(\Omega,\mathcal{O}(K_{\Omega}))$ such that $F=f+o(w_j^{k_j})dw_j$ near $z_j$ for any $j\big\}$, 
  equality 
  \begin{equation}
  	\label{eq:241208b}
  	C_{\Omega,f}=\left(\int_0^{+\infty}c(s)e^{-s}ds\right)\sum_{1\le j\le m}\frac{2\pi|a_j|^2e^{-\alpha_j}}{p_jc_{\beta}(z_j)^{2(k_j+1)}}
  \end{equation}
 holds if and only if the following statements hold:

	$(1)$ $\psi=2\sum_{1\le j\le m}p_jG_{\Omega}(\cdot,z_j)$;
	
	$(2)$ $\varphi+\psi=2\log|g|+2\sum_{1\le j\le m}(k_j+1)G_{\Omega}(\cdot,z_j)+2u$, where $g\in\mathcal{O}(\Omega)$ such that $g(z_j)\not=0$ for any $j\in\{1,2,\ldots,m\}$ and $u$ is a harmonic function on $\Omega$;
	
	$(3)$ $\prod_{1\le j\le m}\chi_{z_j}^{k_j+1}=\chi_{-u}$;
	
	$(4)$ There is a constant $c_0\in\mathbb{C}\backslash\{0\}$ s.t. $\lim\limits_{z\rightarrow z_k}\frac{f}{gP_*\Big(f_u\big(\prod\limits_{1\le j\le m}f_{z_j}^{k_j+1}\big)\big(\sum\limits_{1\le j\le m}p_{j}\frac{d{f_{z_{j}}}}{f_{z_{j}}}\big)\Big)}=c_0$ for any $k\in\{1,2,\ldots,m\}$.
\end{Theorem}

We give an example of Theorem \ref{c:L2-1d-char} in Appendix \ref{appendix}. 

In the $L^2$ extension theorems (see \cite{Ohsawa5,guan-zhou13ap,D2016}), letting the weight $\varphi\equiv0$ for simplicity, the $L^2$ condition on $f$  was set as $$\limsup_{t\rightarrow+\infty}\int_{\{-t-1<\psi<-t\}}|f|^2e^{-\psi}<+\infty,$$ where $\psi$ is the polar function.  For jets extension,  
let $\Omega=\Delta$ be the unit disc in $\mathbb{C}$, $Z_0=\{o\}$, $\varphi\equiv0$, $\psi=2(k+1)\log|w|$ and $f=\sum_{l\ge0}a_lw^l$ on $\Delta$, then
it is clear that $\limsup_{t\rightarrow+\infty}\int_{\{-t-1<\psi<-t\}}|f|^2e^{-\psi}<+\infty$ if and only if $a_l=0$ for any $l<k$. Thus, we assume that all the terms of order $<k_j$ of $f$ at $z_j$ vanish  in the above theorem.

\begin{Remark}\label{rem:1.2}
When the four statements in Theorem \ref{c:L2-1d-char} hold,
$$c_0gP_*\Bigg(f_u\bigg(\prod\limits_{1\le j\le m}f_{z_j}^{k_j+1}\bigg)\bigg(\sum_{1\le j\le m}p_{j}\frac{d{f_{z_{j}}}}{f_{z_{j}}}\bigg)\Bigg)$$
 is the unique  $F\in H^0(\Omega,\mathcal{O}(K_{\Omega}))$ such that $F=f+o(w_j^{k_j})dw_j$ near $z_j$ for any $j$ and
$
 	\int_{\Omega}|F|^2e^{-\varphi}c(-\psi)\leq\left(\int_0^{+\infty}c(s)e^{-s}ds\right)\sum_{1\le j\le m}\frac{2\pi|a_j|^2e^{-\alpha_j}}{p_jc_{\beta}(z_j)^{2(k_j+1)}}.
$ 
\end{Remark}

In the following, we consider the optimal $L^2$ extension problem from arbitrary analytic subsets to $\Omega$.

Let $S:=\{z_j:1\le j< \gamma\}$ be an analytic subset of $\Omega$, where $\gamma\in\mathbb{Z}_{>1}\cup\{+\infty\}$. Let $\psi\in SH^-(\Omega)$ and   $\varphi+\psi\in SH(\Omega)$. Let $c(t)$ is a positive function on $(0,+\infty)$ satisfying $c(t)e^{-t}$ is decreasing on $(0,+\infty)$ and $\int_0^{+\infty}c(t)e^{-t}dt<+\infty.$ Let $k_j$ be a nonnegative integer for any $1\le j<\gamma$. Assume that  $\nu(dd^c\psi,z_j)=\nu(dd^c(\varphi+\psi),z_j)=2(k_j+1)$ and  $$\alpha_j:=(\varphi+\psi-2(k_j+1)G_{\Omega}(\cdot,z_j))(z_j)>-\infty$$ 
for any $j$. Let $w_j$ be a local coordinate on a neighborhood $V_{z_j}\Subset\Omega$ of $z_j$ satisfying $w_j(z_j)=0$ for $j.$

Choosing $f$ be a holomorphic $(1,0)$ form on a neighborhood of $S$ satisfying $f=a_jw_j^{k_j}dw_j$ near $z_j$ for any $j$, where $\sum_{1\le j<\gamma}|a_j|\not=0$,  
the optimal jets $L^2$  extension theorem (see Proposition \ref{p:extension}) shows that there exists an $F\in H^0(\Omega,\mathcal{O}(K_{\Omega}))$ such that $F=f+o(w_j^{k_j})dw_j$ near $z_j$ for any $j$ and 
\begin{equation}
	\label{eq:1218a}
	\int_{\Omega}|F|^2e^{-\varphi}c(-\psi)\leq\left(\int_0^{+\infty}c(s)e^{-s}ds\right)\sum_{1\le j<\gamma}\frac{2\pi|a_j|^2e^{-\alpha_j}}{(k_j+1)c_{\beta}(z_j)^{2(k_j+1)}}.
\end{equation}

The minimal $L^2$ integral $C_{\Omega,f}$ of holomorphic extensions  is defined by
$$\inf\left\{\int_{\Omega}|\tilde F|^2e^{-\varphi}c(-\psi):\tilde F\in H^0(\Omega,\mathcal{O}(K_{\Omega}))\,\&\,\tilde F=f+o(w_j^{k_j})dw_j\,\text{for  $1\le j<\gamma$}\right\}.$$
Especially, when $\gamma=2$, $\varphi\equiv0$, $k_1=0$ and $c\equiv1$, we have $C_{\Omega,f}=\frac{2|a_1|^2}{B_{\Omega}(z_1)}$. Inequality \eqref{eq:1218a} shows that 
\begin{equation}
	\label{eq:231231a}
	C_{\Omega,f}\le\left(\int_0^{+\infty}c(s)e^{-s}ds\right)\sum_{1\le j<\gamma}\frac{2\pi|a_j|^2e^{-\alpha_j}}{(k_j+1)c_{\beta}(z_j)^{2(k_j+1)}},
\end{equation}
which is a generalization of the inequality part of Suita conjecture, i.e., $(c_{\beta}(z_0))^2\le\pi B_{\Omega}(z_0)$.

Using Theorem \ref{c:L2-1d-char} and Theorem \ref{p:infinite-extension}, we obtain a sufficient and necessary condition for inequality \eqref{eq:231231a} to become an equality, which gives a solution to Problem \ref{problem:2} on the open Riemann surfaces.

\begin{Corollary}\label{c:::}
	Equality
	$$C_{\Omega,f}=\left(\int_0^{+\infty}c(s)e^{-s}ds\right)\sum_{1\le j<\gamma}\frac{2\pi|a_j|^2e^{-\alpha_j}}{(k_j+1)c_{\beta}(z_j)^{2(k_j+1)}}$$
	holds if and only if 
	
		$(1)$ $\gamma<+\infty$ and $\psi=2\sum_{1\le j<\gamma}(k_j+1)G_{\Omega}(\cdot,z_j)$;
	
	$(2)$ $\varphi=2\log|g|+2u$, where $g\in\mathcal{O}(\Omega)$ such that $g(z_j)\not=0$ for any $j$ and $u$ is a harmonic function on $\Omega$;
	
	$(3)$ $\prod_{1\le j<\gamma}\chi_{z_j}^{k_j+1}=\chi_{-u}$;
	
	$(4)$  $\lim_{z\rightarrow z_k}\frac{f}{gP_*\Big(f_u\big(\prod\limits_{1\le j<\gamma}f_{z_j}^{k_j+1}\big)\big(\sum\limits_{1\le j<\gamma}(k_j+1)\frac{d{f_{z_{j}}}}{f_{z_{j}}}\big)\Big)}=c_0$ for any $k\in\{1,2,\ldots,\gamma-1\}$, where $c_0\in\mathbb{C}\backslash\{0\}$ is a constant independent of $k$.
\end{Corollary}

Note that $\chi_{z_j}=1$ if and only if $\Omega$  is conformally equivalent to the unit disc less a (possible) closed	set of inner capacity zero (see \cite{suita72}).
When $S$ is a single point set, $e^{-\varphi}c(-\psi)\equiv1$ and $k_1=0$, Corollary \ref{c:::} is the solution of the equality part of Suita conjecture (see \cite{suita72,guan-zhou13ap}). When $S$ is a single point set, Corollary \ref{c:::} can be found in \cite{GMY-concavity2}, whose proof rely on the solution of the (extended) Suita conjecture.

We prove the weighted jets version of Suita conjecture for analytic subsets without using the solution of the (extended) Suita conjecture in this article:
We construct the expression of the minimal extension by the universal covering and Green functions and then perform calculations on open Riemann surfaces (see Section \ref{sec:2}) to establish the sufficiency of the characterization; to prove the necessity part, we repeatedly analyze the concavity of minimal $L^2$ integrals degenerating to linearity, including some properties of the unique minimal extension, the influence of the subharmonicity of $\psi$ and $\varphi+\psi$ on concavity and the interactions between multiple points. 

\section{Some results on open Riemann surfaces}
\label{sec:2}

Let $\Omega$ be an open Riemann surface with a nontrivial Green function $G_{\Omega}$. Let $z_0\in\Omega$.

\begin{Lemma}[see \cite{S-O69}, see also \cite{Tsuji}]
	\label{l:green-sup}Let $w$ be a local coordinate on a neighborhood of $z_0$ satisfying $w(z_0)=0$.  $G_{\Omega}(z,z_0)=\sup_{v\in\Delta(z_0)}v(z)$, where $\Delta(z_0)$ is the subset of  $ SH^-(\Omega)$ such that $v-\log|w|$ is bounded from above near $z_0$.
\end{Lemma}

\begin{Lemma}[see \cite{GY-concavity}]\label{l:G-compact}
	For any open neighborhood $U$ of $z_0$, there exists $t>0$ such that $\{G_{\Omega}(z,z_0)<-t\}$ is a relatively compact subset of $U$.
\end{Lemma}

The following lemma will be used in the proof of Theorem \ref{thm:m-points}.

 \begin{Lemma}[see \cite{GY-concavity}]
	\label{l:cu}
	Let $T$ be a closed positive $(1,1)$ current on $\Omega$. For any open set $U\Subset \Omega$ satisfying $U\cap supp T\not=\emptyset$,  there exists   $\Phi\in SH^-(\Omega)$ satisfying that:
	$(1)$ $i\partial\bar\partial\Phi\leq T$ and $i\partial\bar\partial\Phi\not\equiv0$;
	$(2)$ $\lim_{t\rightarrow0+0}(\inf_{\{G_{\Omega}(z,z_0)\geq-t\}}\Phi(z))=0$;
	$(3)$ $supp (i\partial\bar\partial\Phi)\subset U$ and $\inf_{\Omega\backslash U}\Phi>-\infty$.
\end{Lemma}

Let $Z_0':=\{z_j:1\le j<\gamma \}$ be a discrete subset of $\Omega$, where $\gamma\in\mathbb{Z}_{\geq2}\cup\{+\infty\}$.

\begin{Lemma}
	\label{l:green-sup2} Let $\psi\in SH^-(\Omega)$ such that $\frac{1}{2}\nu(dd^c\psi,z_j)\ge p_j$ for any $j$, where $p_j>0$ is a constant. Then $2\sum_{1\le j< \gamma}p_jG_{\Omega}(\cdot,z_j)\in SH^-(\Omega)$ satisfying that $2\sum_{1\le j<\gamma }p_jG_{\Omega}(\cdot,z_j)\ge\psi$ and $2\sum_{1\le j<\gamma }p_jG_{\Omega}(\cdot,z_j)$ is harmonic on $\Omega\backslash Z_0'$.
\end{Lemma}
\begin{proof}
	It follows from Lemma \ref{l:green-sup} and  Siu's Decomposition Theorem \cite{siu74} that $\psi-2p_1G_{\Omega}(\cdot,z_1)<0\in SH(\Omega)$. Take $\psi_1=\psi-2p_1G_{\Omega}(\cdot,z_1)$, then $\psi_2:=\psi_1-2p_2G_{\Omega}(\cdot,z_2)\in SH^-(\Omega)$. Thus, for any $1\le l<\gamma$, we have $$\psi_l:=\psi-2\sum_{1\le j\le l}p_jG_{\Omega}(\cdot,z_j)\in SH^-(\Omega).$$ 
	 As $\{2\sum_{1\le j\le l}p_jG_{\Omega}(\cdot,z_j)\}$ is decreasing with respect to $l$ and $2\sum_{1\le j\le l}p_jG_{\Omega}(\cdot,z_j)\ge\psi$, we have $2\sum_{1\le j<\gamma }p_jG_{\Omega}(\cdot,z_j)\in SH^-(\Omega)$ and $2\sum_{1\le j<\gamma }p_jG_{\Omega}(\cdot,z_j)\ge\psi$.
	
	Now we prove $2\sum_{1\le j<\gamma }p_jG_{\Omega}(\cdot,z_j)$ is harmonic on $\Omega\backslash Z_0'$. It suffices to prove the case $\gamma=+\infty$. Note that $\sum_{1\le j\le m}G_{\Omega}(\cdot,z_m)$ is harmonic on $\Omega\backslash\{Z'_0\}$ and $\{2\sum_{1\le j\le l}p_jG_{\Omega}(\cdot,z_j)\}$ is decreasing  to $2\sum_{1\le j<\gamma }p_jG_{\Omega}(\cdot,z_j)\ge\psi$. By Harnack's inequality, we know that $2\sum_{1\le j<\gamma }p_jG_{\Omega}(\cdot,z_j)$ is harmonic on $\Omega\backslash Z_0'$.
\end{proof}

\begin{Lemma}
	\label{l:psi=G}
 Let $\psi$ be as in Lemma \ref{l:green-sup2}. Let $l(t)$ is a positive Lebesgue measurable function on $(0,+\infty)$ satisfying $l$ is decreasing and $\int_0^{+\infty}l(t)dt<+\infty$. If $\psi\not\equiv 2\sum_{1\le j<\gamma}p_jG_{\Omega}(\cdot,z_j)$,  there exists a Lebesgue measurable subset $V$ of $\Omega$, such that  $$l(-\psi(z))<l(-2\sum_{1\le j<\gamma}p_jG_{\Omega}(z,z_j))$$ for any  $z\in V$ and $\mu(V)>0$, where $\mu$ is the Lebesgue measure on $\Omega$.
\end{Lemma}

\begin{proof} Taking $U_0\Subset\Omega\backslash\{z_j:2\le j<\gamma \}$ be a neighborhood of $z_1$,
	it follows from Lemma \ref{l:green-sup2} and Lemma \ref{l:G-compact} that there exists $t_0>0$ such that $\{z\in U_0:2\sum_{1\le j<\gamma }p_jG_{\Omega}(z,z_j)<-t_0\}\subset\subset U_0$. As $l$ is decreasing and $\int_0^{+\infty}l(t)dt<+\infty$, then there exists $t_1>t_0$ such that $l(t)<l(t_1)$ holds for any $t>t_1.$
	
	Note that $\psi-2\sum_{1\le j<\gamma }p_jG_{\Omega}(\cdot,z_j)\in SH^-(\Omega)$. As $\psi$ is upper semicontinuous,  $\sup_{z\in\{2\sum_{1\le j<\gamma }p_jG_{\Omega}(\cdot,z_j)\leq-t_1\}\cap U_0}\psi(z)<-t_1.$ Thus there exists  $t_2\in(t_0,t_1)$ such that $$-t_3:=\sup_{z\in\{2\sum_{1\le j<\gamma }p_jG_{\Omega}(\cdot,z_j)\leq-t_2\}\cap U_0}\psi(z)<-t_1.$$  Let $V=\{z\in\Omega:-t_1<2\sum_{1\le j<\gamma }p_jG_{\Omega}(z,z_j)<-t_2\}\cap U_0$, then $\mu(V)>0$. As $l(t)$ is decreasing on $(0,+\infty)$, for any $z\in V$, we have
	$l(-\psi(z))\leq l(t_3)<l(t_1)\leq l(-2\sum_{1\le j<\gamma }p_jG_{\Omega}(\cdot,z_j)).$
	Thus, Lemma \ref{l:psi=G} holds.
\end{proof}

We will construct a sequence of subsurfaces $\{\Omega_l\}$ of $\Omega$ in the following lemma.
\begin{Lemma}
	\label{l:1}There exists a sequence of open Riemann surfaces $\{\Omega_l\}_{l\in\mathbb{Z}^+}$ such that $$z_0\in\Omega_l\Subset\Omega_{l+1}\Subset\Omega,\,\,\,\, \cup_{l\in\mathbb{Z}^+}\Omega_l=\Omega,$$ 
	$\Omega_l$ has a smooth boundary $\partial\Omega_l$ in $\Omega$  and $e^{G_{\Omega_l}(\cdot,z_0)}$ can be smoothly extended to a neighborhood of $\overline{\Omega_l}$ for any $l\in\mathbb{Z}^+$, where $G_{\Omega_l}$ is the Green function of $\Omega_l$. Moreover, $\{{G_{\Omega_l}}(\cdot,z_0)-G_{\Omega}(\cdot,z_0)\}$ is decreasingly convergent to $0$ on $\Omega$.
\end{Lemma}
\begin{proof}
	It follows from the embedding theorem of Stein manifolds (see \cite{hormander}) that there exists an element $v\in\mathcal{O}(\Omega)^{3}$ which defines a one-to-one regular proper map from $\Omega$ into $\mathbb{C}^{3}$. Denote that $\Psi:=v^*(\log|z-v(z_0)|)$. Note that ${\Psi}$ is smooth on $\Omega\backslash\{z_0\}$ and $\{\Psi<t\}\Subset\Omega$ for any $t\in\mathbb{R}$. Using Sard's Theorem, there exists a sequence of increasing numbers $\{t_l\}_{l\in\mathbb{Z}^+}$, such that $\lim_{l\rightarrow+\infty}t_l=+\infty$ and $\Omega'_l:=\{\Psi<t_l\}$ has a smooth boundary in $\Omega$. Then  $z_0\in\Omega'_l\Subset\Omega'_{l+1}\Subset\Omega$ and $\cup_{l\in\mathbb{Z}^+}\Omega'_l=\Omega$.
	
	Note that $e^{G_{\Omega'_l}(\cdot,z_0)}$ is smooth on $\Omega'_l$. Lemma \ref{l:green-sup} shows that $G_{\Omega'_l}(\cdot,z_0)\ge e^{\Psi-t_l}$. As $\lim_{z\rightarrow p}(\Psi(z)-t_l)=0$ for any $p\in\partial\Omega'_l$ and $G_{\Omega'_l}(\cdot,z_0)<0$ on $\Omega$, then $\lim_{z\rightarrow p}G_{\Omega'_l}(z,z_0)=0$ on $\Omega'_l$ for any $p\in\partial{\Omega'_l}$, which implies that $e^{G_{\Omega'_l}(\cdot,z_0)}$ can be continuously extended to $\overline{\Omega'_l}$. There exists $s_1>0$ such that $\Omega_1:=\{z\in\Omega'_1:G_{\Omega'_1}(z,z_0)<-s_1\}$ and  $\partial\Omega_1$ is smooth, and  there exists $s_l>0$ such that $\Omega_l:=\{z\in\Omega'_1:G_{\Omega'_l}(z,z_0)<-s_l\}\Supset\Omega'_{l-1}$ and  $\partial\Omega_l$ is smooth for any $l\in\{2,3,4,\ldots\}$. Note that $G_{\Omega_l}(\cdot,z_0)=G_{\Omega'_l}(\cdot,z_0)+s_l$ for any $l\in\mathbb{Z}^+$. Thus, we have $z_0\in\Omega_l\Subset\Omega_{l+1}\Subset\Omega$, $\cup_{l\in\mathbb{Z}^+}\Omega_l=\Omega$, $\Omega_l$ has a smooth boundary $\partial\Omega_l$ in $\Omega$  and $e^{G_{\Omega_l}(\cdot,z_0)}$ can be smoothly extended to a neighborhood of $\overline{\Omega_l}$ for any $l\in\mathbb{Z}^+$.
	
	It follows from Lemma \ref{l:green-sup} that $G_{\Omega_l}(z,z_0)\ge G_{\Omega_{l+1}}(z,z_0)\ge G_{\Omega}(z,z_0)$ for any $z\in\Omega_l$ and $l\in\mathbb{Z}^+$, which implies that $G_{\Omega}(\cdot,z_0)\le \lim_{l\rightarrow+\infty}G_{\Omega_l}(\cdot,z_0)=:\tilde{G} \in SH^-(\Omega)$. Following from Lemma \ref{l:green-sup}, we have $\tilde{G}=G_{\Omega}(\cdot,z_0)$, which implies that $\{G_{\Omega_l}(\cdot,z_0)-G_{\Omega}(\cdot,z_0)\}$ is decreasingly convergent to $0$ on $\Omega$ since $G_{\Omega_l}(\cdot,z_0)-G_{\Omega}(\cdot,z_0)$ is harmonic on $\Omega_l$ for any $l$.	
	\end{proof}

Let $\Omega_l$ be the open Riemann surface in Lemma \ref{l:1}. We recall a well-known property of $G_{\Omega_l}$. For convenience of readers, we give a proof.
\begin{Lemma}\label{l:di}
Let $h\in C^2(\overline{\Omega_l})\cap  SH(\Omega)$. Then we have
\begin{equation}
	\label{eq:211019a}
	\int_{\partial \Omega_l}hd^cG_{\Omega_l}(\cdot,z_0)\ge\ h(z_0),
\end{equation}
where $d^c=\frac{\partial-\overline{\partial}}{2\sqrt{-1}\pi}$. When $h$ is harmonic, inequality \eqref{eq:211019a} becomes an equality.
\end{Lemma}
\begin{proof}
	Let $\tilde{w}$ be a local coordinate on a neighborhood $\tilde{V}_{z_0}\Subset\Omega_l$ of $z_0$ satisfying $\tilde{w}(z_0)=0$ and $|\tilde{w}|=e^{G_{\Omega_l}(\cdot,z_0)}$ on $\tilde{V}_{z_0}$. It follows from Stokes' theorem that
	\begin{equation}
		\label{eq:211020a}
		\int_{\partial \Omega_l}hd^cG_{\Omega_l}(\cdot,z_0)=\int_{\Omega_l\backslash\overline{B_r}}d(hd^cG_{\Omega_l}(\cdot,z_0))+\int_{\partial B_r}hd^cG_{\Omega_l}(\cdot,z_0),
	\end{equation}
	where $B_r=\{z\in\tilde{V}_{z_0}:|\tilde{w}(z)|<r\}\Subset\tilde{V}_{z_0}$.
Note that $d^cG_{\Omega_l}(\cdot,z_0)=\frac{\partial-\overline{\partial}}{2\sqrt{-1}\pi}\log|\tilde{w}|=\frac{\overline{\tilde{w}}d\tilde{w}-\tilde{w}\overline{d\tilde{w}}}{4\sqrt{-1}\pi r^2}$ on $\partial B_r$. Let $\tilde{w}=re^{\sqrt{-1}\theta}$, where $\theta\in[0,2\pi)$, then we have
\begin{equation}
	\label{eq:211020d}
	\int_{\partial B_r}hd^cG_{\Omega_l}(\cdot,z_0)=\int_{\partial B_r}h\frac{\overline{\tilde{w}}d\tilde{w}-\tilde{w}\overline{d\tilde{w}}}{4\sqrt{-1}\pi r^2}
	=\frac{1}{2\pi}\int_{0}^{2\pi}h(\tilde{w}^{-1}(re^{\sqrt{-1}\theta}))d\theta.
\end{equation}
As $h$ is subharmonic, then inequality \eqref{eq:211020d} implies that
\begin{equation}
	\label{eq:211020e}\int_{\partial B_r}hd^cG_{\Omega_l}(\cdot,z_0)\ge h(z_0)
\end{equation}
for  $r$ is small enough.
It follows from Stokes' theorem that
	\begin{equation}
		\nonumber
		\lim_{r\rightarrow0}\int_{\partial{B_r}}G_{\Omega_l}(\cdot,z_0)d^ch=		\lim_{r\rightarrow0}\log r\int_{B_r}dd^ch=0,
	\end{equation}
	which implies that
		$$\lim_{r\rightarrow0}\int_{\Omega_l\backslash\overline{B_r}}d(G_{\Omega_l}(\cdot,z_0)d^ch)	
		=\int_{\partial\Omega_l}G_{\Omega_l}(\cdot,z_0)d^ch-\lim_{r\rightarrow0}\int_{\partial{B_r}}G_{\Omega_l}(\cdot,z_0)d^ch
		=0.$$
	Note that   $G_{\Omega_l}(\cdot,z_0)$ is harmonic on $\Omega_l\backslash\overline{B_r}$ and $d(G_{\Omega_l}(\cdot,z_0)d^ch)=(dG_{\Omega_l}(\cdot,z_0))\wedge(d^ch)+G_{\Omega_l}(\cdot,z_0)dd^ch=(dh)\wedge(d^cG_{\Omega_l}(\cdot,z_0))+G_{\Omega_l}(\cdot,z_0)dd^ch$. Then
	\begin{equation}
		\nonumber \begin{split}
			\lim_{r\rightarrow0}\int_{\Omega_l\backslash\overline{B_r}}d(hd^cG_{\Omega_l}(\cdot,z_0))
			=&\lim_{r\rightarrow0}\int_{\Omega_l\backslash\overline{B_r}}(dh)\wedge(d^cG_{\Omega_l}(\cdot,z_0))\\
			=&\lim_{r\rightarrow0}\int_{\Omega_l\backslash\overline{B_r}}d(G_{\Omega_l}(\cdot,z_0)d^ch)-G_{\Omega_l}(\cdot,z_0)dd^ch\\
			=&\lim_{r\rightarrow0}\int_{\Omega_l\backslash\overline{B_r}}-G_{\Omega_l}(\cdot,z_0)dd^ch.
		\end{split}
	\end{equation}
	As $G_{\Omega_l}(\cdot,z_0)<0$ on $\Omega_l$ and $h$ is subharmonic on $\Omega_l$, we get
	\begin{equation}
		\label{eq:211020f}\lim_{r\rightarrow0}\int_{\Omega_l\backslash\overline{B_r}}d(hd^cG_{\Omega_l}(\cdot,z_0))\ge0.
	\end{equation}	
	Combining equality \eqref{eq:211020a}, inequality \eqref{eq:211020e} and \eqref{eq:211020f}, we have
	\begin{displaymath}
		\begin{split}
		\int_{\partial \Omega_l}h(d^cG_{\Omega_l})&=\lim_{r\rightarrow0}\Bigg(\int_{\Omega_l\backslash\overline{B_r}}d(hd^cG_{\Omega_l}(\cdot,z_0))+\int_{\partial B_r}hd^cG_{\Omega_l}\Bigg)
		\\&\ge h(z_0),
		\end{split}
	\end{displaymath}
	which is inequality \eqref{eq:211019a}. When $h$ is harmonic,  inequality \eqref{eq:211020e} and inequality \eqref{eq:211020f} become equalities, which implies that inequality \eqref{eq:211019a} becomes equality. 
	\end{proof}

Without loss of generality,  assume that $\{z_1,z_2,\ldots,z_m\}\subset\Omega_l$ for any $l\in\mathbb{Z}^+$. Let $p_j>2$ be a real number for any $j\in\{1,2,\ldots,m\}$.
Denote that $$\mathcal{G}:=2\sum_{1\le j\le m}p_jG_{\Omega}(\cdot,z_j)\,\text{on $\Omega$, }\,\,\,\, \mathcal{G}_l:=2\sum_{1\le j\le m}p_jG_{\Omega_l}(\cdot,z_j)\,\text{on $\Omega_l$.}$$  

\begin{Lemma}
	\label{l:2}
	There exists a subsequence  $\{e^{\mathcal{G}_{l_n}}\}_{n\in\mathbb{Z}^+}$ of $\{e^{\mathcal{G}_l}\}_{l\in\mathbb{Z}^+}$ such that $\{e^{\mathcal{G}_{l_n}}\}$, $\{de^{\mathcal{G}_{l_n}}\}$ and $\{\partial\overline{\partial}e^{\mathcal{G}_{l_n}}\}$ are uniformly convergent to $e^\mathcal{G}$, $de^\mathcal{G}$ and $\partial\overline{\partial}e^\mathcal{G}$ on any compact subset of $\Omega$, respectively.
\end{Lemma}
\begin{proof}
Using the diagonal method, it suffices to prove this lemma locally. Choosing any $p\in\Omega$,  there exists a neighborhood $U\Subset\Omega$ of $p$, such that $U$ is conformally equivalent to the unit disc and $U\cap\{z_1,z_2,\ldots,z_m\}$ has at most one point.
	
If $U\cap\{z_1,z_2,\ldots,z_m\}=\emptyset$,  then there exists  $f_l\in\mathcal{O}(U)$ such that $|f_l|^2=e^{\mathcal{G}_{l}}$ for any $l$ and $f_0\in\mathcal{O}(U)$ such that $|f_0|^2=e^{\mathcal{G}}$.
Lemma \ref{l:1} shows that $\{|f_l|\}_{l\in\mathbb{Z}^+}$ is decreasingly convergent to $|f_0|$, which implies that there exists a subsequence of $\{f_{l_n}\}_{n\in\mathbb{Z}^+}$ of $\{f_l\}_{l\in\mathbb{Z}^+}$. Then $\{f_{l_n}\}$ and $\{df_{l_n}\}$ are uniformly convergent to $f_0$ and  $df_0$ on any compact subset of $U$, respectively.
Note that $de^{\mathcal{G}_{l_n}}=d(f_{l_n}\overline{f_{l_n}})=(\partial f_{l_n})f_{l_n}+f_l\overline{\partial f_{l_n}}$ and $\partial\overline{\partial}e^{\mathcal{G}_{l_n}}=\partial f_{l_n}\wedge\overline{\partial f_{l_n}}$. Then we have $\{e^{\mathcal{G}_{l_n}}\}$, and $\{\partial\overline{\partial}e^{\mathcal{G}_{l_n}}\}$  are uniformly convergent to $e^\mathcal{G}$, $de^\mathcal{G}$ and $\partial\overline{\partial}e^\mathcal{G}$ on any compact subset of $U$, respectively.
	
If $U\cap\{z_1,z_2,\ldots,z_m\}=\{z_{j_0}\}$ (without loss of generality, assume that $j_0=1$),  then there exists  $f_l\in\mathcal{O}(U)$ such that $|f_l|^2=e^{\frac{\mathcal{G}_{l}}{p_1}}$ for any $l$ and  $f_0\in\mathcal{O}(U)$ such that $|f_0|^2=e^{\frac{\mathcal{G}}{p_1}}$.
Lemma \ref{l:1} shows that $\{|f_l|\}_{l\in\mathbb{Z}^+}$ is decreasingly convergent to $|f_0|$. Then there exists a subsequence of $\{f_{l_n}\}_{n\in\mathbb{Z}^+}$ of $\{f_l\}_{l\in\mathbb{Z}^+}$, which  satisfies that $\{f_{l_n}\}$ and $\{df_{l_n}\}$ are uniformly convergent to $f_0$  and $df_0$ on any compact subset of $U$, respectively.
As $p_1>2$ and $e^{\mathcal{G}_{l_n}}=|f_{l_n}|^{2p_1}$, by a direct calculation, we have
	$$de^{\mathcal{G}_{l_n}}=d(|f_{l_n}|^{2p_1})=p_1|f_{l_n}|^{2(p_1-1)}(\overline{f_{l_n}}\partial f_{l_n}+f_l\overline{\partial f_{l_n}})$$ and
	 $$\partial\overline{\partial}e^{\mathcal{G}_{l_n}}=\partial\overline{\partial}{|f_{l_n}|^{2p_1}}=p_1^2|f_{l_n}|^{2(p_1-1)}\partial f_{l_n}\wedge\overline{\partial f_{l_n}}.$$ Then $\{e^{\mathcal{G}_{l_n}}\}$,  $\{de^{\mathcal{G}_{l_n}}\}$ and $\{\partial\overline{\partial}e^{\mathcal{G}_{l_n}}\}$ are uniformly convergent to $e^\mathcal{G}$, $de^\mathcal{G}$ and $\partial\overline{\partial}e^\mathcal{G}$ on any compact subset of $U$, respectively.
		
	Thus, Lemma \ref{l:2} holds.
\end{proof}

In the following, let $\{\Omega_l\}$ and $\{e^{\mathcal{G}_{l}}\}$ be as the $\{\Omega_{l_n}\}$ and $\{e^{\mathcal{G}_{l_n}}\}$ in Lemma \ref{l:2}, respectively. 
Let us calculate the integral $\sqrt{-1}\int_{\Omega}\partial\overline\partial e^{\mathcal{G}}$.
 \begin{Lemma}
	\label{l:4}
	$\sqrt{-1}\int_{\Omega}\partial\overline\partial e^{\mathcal{G}}=2\pi\sum_{1\le j\le m}p_j$, where $\mathcal{G}=2\sum_{1\le j\le m}p_jG_{\Omega}(\cdot,z_j)$.
\end{Lemma}
\begin{proof}
Note that $\sqrt{-1}\partial\overline{\partial}=\pi dd^c$ and $\mathcal{G}_l=0$ on $\partial\Omega_l$. Using Stokes's theorem and Lemma \ref{l:di}, we have
\begin{equation}
	\label{eq:241201a}
	\begin{split}
		\sqrt{-1}\int_{\Omega_l}\partial\overline{\partial}e^{\mathcal{G}_l}&=\pi\int_{\Omega_l}dd^ce^{\mathcal{G}_l}=\pi\int_{\partial\Omega_l}d^ce^{\mathcal{G}_l}=\pi\int_{\partial\Omega_l}e^{\mathcal{G}_l}(d^c\mathcal{G}_l)\\
		&=2\pi\sum_{1\le j\le m}p_j\int_{\partial\Omega_l}d^cG_{\Omega_l}(\cdot,z_j)
		=2\pi\sum_{1\le j\le m}p_j
	\end{split}
\end{equation}
and 
\begin{equation}
	\label{eq:211020h}\begin{split}
	\int_{\Omega_l}dd^c(e^{\mathcal{G}}-e^{\mathcal{G}_l})&=\int_{\partial\Omega_l}d^c(e^{\mathcal{G}}-e^{\mathcal{G}_l})=\int_{\partial\Omega_l}e^{G}d^c(\mathcal{G}_l+(\mathcal{G}-\mathcal{G}_l))-e^{\mathcal{G}_l}d^c\mathcal{G}_l\\
	&=\int_{\partial\Omega_l}(e^\mathcal{G}-e^{\mathcal{G}_l})d^c\mathcal{G}_l+\int_{\partial\Omega_l}e^{\mathcal{G}-\mathcal{G}_l}d^c(\mathcal{G}-\mathcal{G}_l).
	\end{split}
\end{equation}
As $\mathcal{G}-\mathcal{G}_l$ is harmonic on $\Omega_l$, then $e^{\mathcal{G}-\mathcal{G}_l}\in SH(\Omega_l)$, which implies that
\begin{equation}
	\label{eq:211020i}\begin{split}
	\int_{\partial\Omega_l}e^{\mathcal{G}-\mathcal{G}_l}d^c(\mathcal{G}-\mathcal{G}_l)=\int_{\partial\Omega_l}d^ce^{\mathcal{G}-\mathcal{G}_l}
=\int_{\Omega_l}dd^ce^{\mathcal{G}-\mathcal{G}_l}\ge0.
	\end{split}\end{equation}
Note that $\mathcal{G}_l=0$ on $\partial\Omega_l$ and $e^{\mathcal{G}-\mathcal{G}_l}$ is subharmonic on $\Omega_l$. Using Lemma \ref{l:di} and $\mathcal{G}_l=2\sum_{1\le j\le m}p_jG_{\Omega}(\cdot,z_j)$, we have
\begin{equation}
	\begin{split}
		\label{eq:211020j}
		\int_{\partial\Omega_l}(e^\mathcal{G}-e^{\mathcal{G}_l})(d^c\mathcal{G}_l)=\int_{\partial\Omega_l}(e^{\mathcal{G}-\mathcal{G}_l}-1)(d^c\mathcal{G}_l)\ge2\sum_{1\le j\le m}p_j\left(e^{\mathcal{G}-\mathcal{G}_l}(z_j)-1\right).
	\end{split}
\end{equation}
As  $\{\mathcal{G}-\mathcal{G}_l\}$ is increasingly convergent to $0$ on $\Omega$, inequality \eqref{eq:211020j} implies that
\begin{equation}
	\label{eq:211020k}
		\liminf_{l\rightarrow+\infty}\int_{\partial\Omega_l}(e^\mathcal{G}-e^{\mathcal{G}_l})(d^c\mathcal{G}_l)\ge0.
\end{equation}
Combining equality \eqref{eq:241201a}, \eqref{eq:211020h} and inequality \eqref{eq:211020i}, \eqref{eq:211020k}, we obtain
\begin{equation}
	\nonumber \begin{split}
	\sqrt{-1}\int_{\Omega}\partial\overline{\partial}e^\mathcal{G}=&\pi\lim_{l\rightarrow+\infty}\int_{\Omega_l}dd^ce^\mathcal{G}
	=\pi\lim_{l\rightarrow+\infty}\int_{\Omega_l}dd^c(e^{\mathcal{G}}-e^{\mathcal{G}_l})+2\pi\sum_{1\le j\le m}p_j\\
	=&\pi\lim_{l\rightarrow+\infty}(\int_{\partial\Omega_l}(e^\mathcal{G}-e^{\mathcal{G}_l})d^c\mathcal{G}_l+\int_{\partial\Omega_l}e^{\mathcal{G}-\mathcal{G}_l}d^c(\mathcal{G}-\mathcal{G}_l))\\
	&+2\pi\sum_{1\le j\le m}p_j\\
	\ge&2\pi\sum_{1\le j\le m}p_j.
	\end{split}
\end{equation}
It follows from Fatou's Lemma and $\int_{\Omega_l}\sqrt{-1}\partial\overline\partial e^{\mathcal{G}_l}=2\pi\sum_{1\le j\le m}p_j$ that
\begin{displaymath}
	\begin{split}
		\sqrt{-1}\int_{\Omega}\partial\overline{\partial}e^\mathcal{G}&=\int_{\Omega}\lim_{l\rightarrow+\infty}\sqrt{-1}\partial\overline\partial e^{\mathcal{G}_l}\le\liminf_{l\rightarrow+\infty}\int_{\Omega_l}\sqrt{-1}\partial\overline\partial e^{\mathcal{G}_l}=2\pi\sum_{1\le j\le m}p_j.
	\end{split}
\end{displaymath}

Thus, Lemma \ref{l:4} holds.
\end{proof}

We present an orthogonal property of $\partial e^{\mathcal{G}}$ in the following lemma.
\begin{Lemma}
	\label{l:5}For any  $\beta\in H^0(\Omega,\mathcal{O}(K_{\Omega}))$ satisfying $\int_{\Omega}|\beta|^2<+\infty$,  $\int_{\Omega}\partial e^{\mathcal{G}} \wedge\overline\beta=0$.
\end{Lemma}

\begin{proof}
 Let $P:\Delta\rightarrow\Omega$ be the universal covering from unit disc $\Delta$ to $\Omega$.
There exists $f_{z_j}\in\mathcal{O}(\Delta)$, such that $|f_{z_j}(z)|=P^{*}e^{G_{\Omega}(z,z_j)}$ for any $j\in\{1,2,\ldots,m\}$. As $p_{j}>2$  for any $j\in\{1,2,\ldots,m\}$, by a direct calculation,  we have
\begin{equation}
\label{eq:211026h}
\begin{split}
	&\partial\overline{\partial}e^{2\sum_{1\le j\le m}p_jG_{\Omega}(\cdot,z_j)}\\
	=&\partial\Bigg(\sum_{1\le j_1\le m}e^{2\sum_{1\le j\le m}p_jG_{\Omega}(\cdot,z_j)-2G_{\Omega}(\cdot,z_{j_1})}p_{j_1}P_*(f_{z_{j_1}}\overline{\partial f_{z_{j_1}}})\Bigg)	\\
	=&\sum_{1\le j_1\le m}\sum_{j_2\not=j_1}	e^{2\sum_{1\le j\le m}p_jG_{\Omega}(\cdot,z_j)}p_{j_1}p_{j_2}P_*\left(\frac{\partial f_{z_{j_2}}}{f_{z_{j_2}}}\wedge\frac{\overline{\partial f_{z_{j_1}}}}{\overline{f_{z_{j_1}}}}\right)\\
	&+\sum_{1\le j_1\le m}e^{2\sum_{1\le j\le m}p_jG_{\Omega}(\cdot,z_j)}p_{j_1}^2P_*\left(\frac{{\partial f_{z_{j_1}}}}{{f_{z_{j_1}}}}\wedge\frac{\overline{\partial f_{z_{j_1}}}}{\overline{f_{z_{j_1}}}}\right)\\
		=&\frac{1}{\sqrt{-1}} P_*(\prod\limits_{1\le j\le m}|f_{z_j}|^{2p_j})\Bigg|P_*\Bigg(\sum_{1\le j\le m}p_j\frac{df_{z_j}}{f_{z_j}}\Bigg)\Bigg|^2
	\end{split}\end{equation}
	and
	\begin{equation}
	\nonumber
	\begin{split}
		\partial e^{2\sum_{1\le j\le m}p_jG_{\Omega}(\cdot,z_j)}&=\sum_{1\le k\le m}e^{2\sum_{1\le j\le m}p_jG_{\Omega}(\cdot,z_j)-2G_{\Omega}(\cdot,z_k)}p_kP_*(\overline{f_{z_k}}\partial f_{z_k})\\
		&=e^{2\sum_{1\le j\le m}p_jG_{\Omega}(\cdot,z_j)}P_*\Bigg(\sum_{1\le j\le m}p_j\frac{df_{z_k}}{f_{z_k}}\Bigg),
	\end{split}
\end{equation} which implies that
$|\partial e^\mathcal{G} |^2=\sqrt{-1}e^\mathcal{G} \partial\overline\partial e^\mathcal{G}.$
Note that $e^{\mathcal{G}}\in SH(\Omega)$ satisfying $e^\mathcal{G}\le1$.
Lemma \ref{l:4} shows that
\begin{equation}
	\nonumber
	\int_{\Omega}|\partial e^\mathcal{G}|^2=\sqrt{-1} \int_{\Omega}e^\mathcal{G}(\partial\overline\partial e^\mathcal{G})\le\sqrt{-1} \int_{\Omega}\partial\overline\partial e^\mathcal{G}=2\pi\sum_{1\le j\le m}p_j.
\end{equation}
Similarly, we have $
	\int_{\Omega_l}|\partial e^{\mathcal{G}_l}|^2\le2\pi\sum_{1\le j\le m}p_j.
$
For any $\epsilon>0$, there exists $l_1>0$ such that
$
	\int_{\Omega\backslash\overline{\Omega_{l_1}}}|\beta|^2<\epsilon^2.
$
As $\{\partial e^{\mathcal{G}_l}\}$ is uniformly convergent to $\partial e^\mathcal{G}$ on any compact subset of $\Omega$, then there exists $M_1>l_1$ such that for any $l\ge M_1$, we have
$
	\left|\int_{{\Omega_{l_1}}}(\partial e^\mathcal{G}-\partial e^{\mathcal{G}_l})\wedge\overline{\beta}\right|<\epsilon.
$
Hence, we have
\begin{equation}
\nonumber
	\begin{split}
		&\left|\int_{\Omega}\partial e^\mathcal{G}\wedge\overline{\beta}-\int_{\Omega_l}\partial e^{\mathcal{G}_l}\wedge\overline{\beta}\right|\\
		\le&\left|\int_{\Omega\backslash\overline{\Omega_{l_1}}}\partial e^\mathcal{G}\wedge\overline{\beta}\right|+\left|\int_{\Omega_l\backslash\overline{\Omega_{l_1}}}\partial e^{\mathcal{G}_l}\wedge\overline{\beta}\right|+\left|\int_{{\Omega_{l_1}}}(\partial e^\mathcal{G}-\partial e^{\mathcal{G}_l})\wedge\overline{\beta}\right|\\
		\le&\left(\int_{\Omega}|\partial e^\mathcal{G}|^2\right)^{\frac{1}{2}}\left(\int_{\Omega\backslash\overline{\Omega_{l_1}}}|\beta|^2\right)^{\frac{1}{2}}+\left(\int_{\Omega_l}|\partial e^{\mathcal{G}_l}|^2\right)^{\frac{1}{2}}\left(\int_{\Omega\backslash\overline{\Omega_{l_1}}}|\beta|^2\right)^{\frac{1}{2}}+\epsilon\\
		\le&\Bigg(2\bigg(2\pi\sum_{1\le j\le m}p_j\bigg)^{\frac{1}{2}}+1\Bigg)\epsilon.
	\end{split}
\end{equation}
As $d\overline{\beta}=0$ and $e^{\mathcal{G}_l}=1$ on $\partial\Omega_l$, it follows from  Stokes' theorem that
\begin{equation}
	\nonumber
	\begin{split}
		\int_{\Omega_l}\partial e^{\mathcal{G}_l}\wedge\overline{\beta}=	\int_{\Omega_l}d(e^{\mathcal{G}_l}\wedge\overline{\beta})
		=\int_{\partial\Omega_l}e^{\mathcal{G}_l}\overline{\beta}
		=\int_{\Omega_l}d\overline{\beta}
		=0.
	\end{split}
\end{equation}
By the arbitrariness of $\epsilon$, we have
$\int_{\Omega}\partial e^\mathcal{G}\wedge\overline{\beta}=0.$
Thus, Lemma \ref{l:5} holds.
\end{proof}

\section{Optimal jets $L^2$ extension on open Riemann surfaces}

In this section, we prove an optimal jets $L^2$ extension theorem on open Riemann surfaces (Proposition \ref{p:extension}). Before presenting that, we recall two lemmas.

\begin{Lemma}[see \cite{GY-concavity,guan-zhou13ap}] \label{lem:L2} Let $c$ be a nonnegative function on $(0,+\infty)$ such that $c(t)e^{-t}$ is decreasing on $(0,+\infty)$.
Let $M$ be a Stein manifold.
Let $\psi\in PSH^-(M)$, and let
 $\varphi\in PSH(M)$. Take any $B\in(0,+\infty)$ and $t_{0}\geq 0$.
Let $F\in H^0(\{\psi<-t_{0}\},\mathcal{O}(K_{M}))$
satisfying
$
\int_{K\cap\{\psi<-t_{0}\}}|F|^{2}<+\infty
$
for any  $K\Subset M$.
Then there exists 
 $\tilde{F}\in H^0(M,\mathcal{O}(K_{M}))$, such that
\begin{equation}
\nonumber
\begin{split}
\int_{M}&|\tilde{F}-(1-b_{t_0,B}(\psi))F|^{2}e^{-\varphi+v_{t_0,B}(\psi)}c(-v_{t_0,B}(\psi))\leq C\int_{0}^{t_{0}+B}c(t)e^{-t}dt
\end{split}
\end{equation}
where
$b_{t_0,B}(t):=\int_{-\infty}^{t}\frac{1}{B}\mathbb{I}_{\{-t_{0}-B< s<-t_{0}\}}ds$,
$v_{t_0,B}(t):=\int_{-t_0}^{t}b_{t_0,B}(s)ds-t_0$ and $C:=\int_{M}\frac{1}{B}\mathbb{I}_{\{-t_{0}-B<\psi<-t_{0}\}}|F|^{2}e^{-\varphi}$.
\end{Lemma}

\begin{Lemma}[see \cite{GY-concavity}]
	\label{l:converge}
	Let $M$ be a complex manifold, and $S$ be an analytic subset of $M$.  	
	Let $\{g_j\}$ be a sequence of nonnegative Lebesgue measurable functions on $M$, which satisfies that $g_j$ are a.e. convergent to a function $g$ on  $M$ when $j\rightarrow+\infty$. Assume that for any compact subset $K$ of $M\backslash S$, there exist $s_K\in(0,+\infty)$ such that
	$\int_{K}{g_j}^{-s_K}dV_M<+\infty$
	 for any $j$, where $dV_M$ is a continuous volume form on $M$.
	
 Let $\{F_j\}\subset  H^0(M,\mathcal{O}(K_{M}))$. Assume that $\liminf_{j\rightarrow+\infty}\int_{M}|F_j|^2g_jdV_M\leq C$, where $C>0$ is a constant. Then there exists a subsequence of $\{F_{j}\}$ uniformly convergent to an $F\in H^0(M,\mathcal{O}(K_{M}))$ on any compact subset of $M$, and 
 $$\int_{M}|F|^2gdV_M\leq C.$$
\end{Lemma}

Let $\Omega$  be an open Riemann surface, which admits a nontrivial Green function $G_{\Omega}$.
Let $Z_0':=\{z_j:1\le j<\gamma \}$ be a discrete subset of $\Omega$, where $\gamma\in\mathbb{Z}_{\geq2}\cup\{+\infty\}$. Let $\psi\in PSH^-(\Omega)$  satisfy $p_j:=\frac{1}{2}\nu(dd^c\psi,z_j)>0$, and let $\varphi$ be a  function on $\Omega$ such that $\varphi+\psi\in PSH(\Omega)$. By the Weierstrass Theorem on open Riemann surface (see \cite{OF81}) and  Siu's Decomposition Theorem \cite{siu74}, then we have
\begin{equation*}
	\varphi+\psi=2\log|g_0|+2u_0,
\end{equation*}
where $g_0\in\mathcal{O}(\Omega)$ and $u_0\in SH(\Omega)$ such that $\nu(dd^cu,z)\in[0,1)$ for all $z\in\Omega$.

Let $w_j$ be a local coordinate on a neighborhood $V_{z_j}\Subset\Omega$ of $z_j$ satisfying $w_j(z_j)=0$ for $z_j\in Z'_0$, where $V_{z_j}\cap V_{z_k}=\emptyset$ for any $j\not=k$. Denote that $V_0:=\cup_{1\le j<\gamma}V_{z_j}$.
Let $f\in H^0(V_0,\mathcal{O}(K_{\Omega}))$. Denote $f=d_{1,j}w_j^{k_{1,j}}h_{1,j}dw_{j}$ and $g_0=d_{2,j}w_{j}^{k_{2,j}}h_{2,j}$ on $V_{z_j}$, where $d_{i,j}\not=0$ are constants, $k_{1,j}$ and $k_{2,j}$ are nonnegative integers, and $h_{i,j}\in\mathcal{O}(V_{z_j})$ satisfying $h_{i,j}(z_j)=1$ for $i\in\{1,2\}$ and $1\le j<\gamma$.

Denote that $I_0:=\{j:1\le j<\gamma\,\&\, k_{1,j}+1-k_{2,j}\le0\}$. Let $c(t)$ be a positive measurable function on $(0,+\infty)$ satisfying $c(t)e^{-t}$ is decreasing and $\int_0^{+\infty}c(t)e^{-t}<+\infty$. Using Lemma \ref{lem:L2}, we give an optimal jets $L^2$ extension theorem on open Riemann surfaces.
\begin{Proposition}\label{p:extension}
	Assume $k_{1,j}+1=k_{2,j}$ and $u_0(z_j)>-\infty$ for $j\in I_0$. Then there exists an ${F}\in H^0(\Omega,\mathcal{O}(K_{\Omega}))$ such that ${F}=f+o(w_j^{k_{2,j}-1})dw_j$ near $z_j$  for any $j$ and
	\begin{equation}
		\label{eq:241203b}
		\int_{\Omega}|{F}|^2e^{-\varphi}c(-\psi)\le\left(\int_0^{+\infty}c(t)e^{-t}dt\right)\sum_{j\in I_0}\frac{2\pi |d_{1,j}|^2e^{-2u_0(z_j)}}{p_j|d_{2,j}|^2}.
	\end{equation}
\end{Proposition}

\begin{proof}
As $c(t)e^{-t}$ is decreasing on $(0,+\infty)$,  following from Lemma \ref{l:green-sup2}  we have $\psi\leq\tilde\psi:=2\sum_{1\le j<\gamma }p_jG_{\Omega}(\cdot,z_j)$ and
$e^{-\varphi}c(-\psi)\leq e^{-(\varphi+\psi-\tilde\psi)}c(-\tilde\psi).$ Thus, we can assume that $\psi=\tilde\psi=2\sum_{1\le j<\gamma }p_jG_{\Omega}(\cdot,z_j)$.

The following remark shows that it suffices to consider the case $\gamma<+\infty$.
\begin{Remark}\label{r:finite}Let $\Omega_l$ be as in 
 Lemma \ref{l:1}. Note that $Z_l:=\Omega_l\cap Z_0'$ is a set of finite points. Denote that $\psi_l:=2\sum_{z_j\in Z_l}p_jG_{\Omega_l}(\cdot,z_j)$, $\varphi_l=\varphi+\psi-\psi_l$  on $\Omega_l$, and $I_l:=I_0\cap\{j:z_j\in Z_l\}$.
Assume Proposition \ref{p:extension} holds for the case $\gamma<+\infty$. Then there exists  $F_l\in H^0(\Omega_l,\mathcal{O}(K_{\Omega_l}))$  such that $F_l=f+o(w_j^{k_{2,j}-1})dw_j$ near $z_j$  for any $z_j\in Z_l$ and
\begin{equation}\label{eq:241203a}
		\int_{\Omega_l}|F_l|^2e^{-\varphi}c(-\psi)
		\le\int_{\Omega_l}|F_l|^2e^{-\varphi_l}c(-\psi_l)
		\leq\left(\int_0^{+\infty}c(t)e^{-t}dt\right)\sum_{j\in I_l}\frac{2\pi|d_{1,j}|^2 e^{-2u_0(z_j)}}{p_j|d_{2,j}|^2}
\end{equation}
since $\psi<\psi_l$ and $c(t)e^{-t}$ is decreasing on $(0,+\infty)$.
Note that $\psi$ is smooth on $\Omega\backslash Z_0'$.
 For any compact subset $K$ of $\Omega\backslash Z_0'$, there exists $s_K>0$ such that $\int_{K}e^{-s_K\psi}dV_{\Omega}<+\infty$, where $dV_{\Omega}$ is a continuous volume form on $\Omega$.  Then  we have
$$\int_{K}\left(\frac{e^{\varphi}}{c(-\psi)}\right)^{s_K}dV_{\Omega}=\int_{K}\left(\frac{e^{\varphi+\psi}}{c(-\psi)}\right)^{s_K}e^{-s_K\psi}dV_{\Omega}\leq C\int_{K}e^{-s_K\psi}dV_{\Omega}<+\infty.$$
By Lemma \ref{l:converge} and the diagonal method, there exists a subsequence of $\{F_l\}$, denoted still by $\{F_l\}$, which is uniformly convergent to an  $ F\in H^0(\Omega,\mathcal{O}(K_{\Omega}))$  on any compact subset of $\Omega$. Thus, ${F}=f+o(w_j^{k_{2,j}-1})dw_j$ near $z_j$  for  $1\le j<\gamma$, inequality \eqref{eq:241203b} holds by letting $l\rightarrow+\infty$ in inequality \eqref{eq:241203a}.
\end{Remark}

Now, we prove the case  that $\gamma=m+1$. Without loss of generality, assume $I_0=\{1,2,\ldots,m_1\}$, where  $m_1<m$  ($I_0=\emptyset$ if and only if $m_1=0$).

As $\Omega$ is a Stein manifold, then there exist $u_l\in SH(\Omega)\cap C^{\infty}(\Omega)$  decreasingly convergent to $u_0$ with respect to $l$ (see \cite{FN80}).
By Lemma \ref{l:G-compact},  $\{\psi<-t_0\}\Subset V_0$ for some $t_0>0$, then $\int_{\{\psi<-t_0\}}|f|^2<+\infty$. Using Lemma \ref{lem:L2}, we obtain an $F_{l,t}\in H^0(\Omega,\mathcal{O}(K_{\Omega}))$ satisfying
\begin{equation}
	\label{eq:211008d}
	\begin{split}
&\int_{\Omega}|F_{l,t}-(1-b_{t,1}(\psi))f|^{2}e^{-2\log|g_0|-2u_l+v_{t,1}(\psi)}c(-v_{t,1}(\psi))\\
\leq& \left(\int_{0}^{t+1}c(s)e^{-s}ds\right) \int_{\Omega}\mathbb{I}_{\{-t-1<\psi<-t\}}|f|^2e^{-2\log|g_0|-2u_l},
\end{split}
\end{equation}
where $t\ge t_0$. Note that $b_{t,1}(s)=0$ for large enough $s$, then $(F_{l,t}-f,z_j)\in(\mathcal{O}(K_{\Omega})\otimes\mathcal{I}(2\log|g_0|))_{z_j}$ i.e. $F_{l,t}=f+o(w_j^{k_{2,j}-1})dw_{j}$ near $z_j$ for any $j\in\{1,2,\ldots,m\}$.

For any $\epsilon>0$, there exists $t_1>t_0$, such that

$(1)$ $\sup_{z\in\{\psi<-t_1\}\cap V_{z_j}}|h_1(z)-h_1(z_j)|<\epsilon$ for any $j\in\{1,2,\ldots,m\}$, where $h_1$ is a smooth function on $V_0$ satisfying that $h_1|_{V_{z_j}}=\psi-2p_j\log|w_j|$;

$(2)$ $\sup_{z\in\{\psi<-t_1\}\cap V_{z_j}}|\frac{d_{1,j}h_{1,j}}{d_{2,j}h_{2,j}}(z)|\le(|\frac{d_{1,j}}{d_{2,j}}|+\epsilon)$ for any $j\in\{1,2,\ldots,m\}$;

$(3)$ $\sup_{z\in\{\psi<-t_1\}\cap V_{z_j}}2|u_l(z)-u_l(z_j)|<\epsilon$ for any $j\in\{1,2,\ldots,m\}$.

Note that $k_{1,j}-k_{2,j}=-1$ for $1\le j\le m_1$ and $k_{1,j}-k_{2,j}>-1$ for $m_1<j\le m$. 
\begin{equation}
\nonumber
\begin{split}
	&\limsup_{t\rightarrow+\infty}\int_{\Omega}\mathbb{I}_{\{-t-1<\psi<-t\}}|f|^2e^{-2\log|g_0|-2u_l}\\
	\le&\limsup_{t\rightarrow+\infty}\sum_{1\le j\le m}\int_{\{-t-1-\epsilon<2p_j\log|w_j|+h_1(z_j)<-t+\epsilon\}}(|\frac{d_{1,j}}{d_{2,j}}|+\epsilon)^2|w_j|^{2(k_{1,j}-k_{2,j})}e^{-2u_l(z_j)+\epsilon}\\
	\le&\sum_{1\le j\le m}4\pi(|\frac{d_{1,j}}{d_{2,j}}|+\epsilon)^2e^{-2u_l(z_j)+\epsilon}\limsup_{t\rightarrow+\infty}\int_{e^{-\frac{t+1+\epsilon+h_1(z_j)}{2p_j}}}^{e^{-\frac{t-\epsilon+h_1(z_j)}{2p_j}}}r^{2(k_{1,j}-k_{2,j})+1}dr\\
	=&\sum_{1\le j\le m_1}2\pi(|\frac{d_{1,j}}{d_{2,j}}|+\epsilon)^2e^{-2u_l(z_j)+\epsilon}\frac{1+2\epsilon}{p_j}.
\end{split}
\end{equation}
Letting $\epsilon\rightarrow+\infty$, we have
\begin{equation}
\nonumber
\limsup_{t\rightarrow+\infty}\int_{\Omega}\mathbb{I}_{\{-t-1<\psi<-t\}}|f|^2e^{-2\log|g_0|-2u_l}
	\le\sum_{1\le j\le m_1}2\pi\frac{|d_{1,j}|^2}{p_j|d_{2,j}|^2}e^{-2u_l(z_j)}
	<+\infty.
\end{equation}
As $v_{t,1}(\psi)\ge\psi$ and $c(t)e^{-t}$ is decreasing, combining inequality \eqref{eq:211008d},  we have
\begin{equation}
	\label{eq:211009c}
	\begin{split}
&\limsup_{t\rightarrow+\infty}\int_{\Omega}|F_{l,t}-(1-b_{t,1}(\psi))f|^{2}e^{-2\log|g_0|-2u_l+\psi}c(-\psi)\\
\le&\limsup_{t\rightarrow+\infty}\int_{\Omega}|F_{l,t}-(1-b_{t,1}(\psi))f|^{2}e^{-2\log|g_0|-2u_l+v_{t,1}(\psi)}c(-v_{t,1}(\psi))\\
\leq& \limsup_{t\rightarrow+\infty}\left(\int_{0}^{t+1}c(s)e^{-s}ds\right) \int_{\Omega}\mathbb{I}_{\{-t-1<\psi<-t\}}|f|^2e^{-2\log|g_0|-2u_l}\\
\leq&\left(\int_0^{+\infty}c(s)e^{-s}ds\right)\sum_{1\le j\le m_1}2\pi\frac{|d_{1,j}|^2}{p_j|d_{2,j}|^2}e^{-2u_l(z_j)}\\
	<&+\infty.
\end{split}
\end{equation}
Denote that $Y:=\{z\in\Omega:g_0(z)=0\}$. For any open set $K\Subset\Omega\backslash Y$, it follows from $b_{t,1}(s)=1$ for any $s\ge -t$ and $c(s)e^{-s}$ is decreasing with respect to $s$ that 
$$\int_{K}|(1-b_{t,1}(\psi))f|^2e^{-2\log|g_0|-2u_l+\psi}c(-\psi)\le C_K\int_{\{\psi<-t_1\}}|f|^2<+\infty$$
 for any $t>t_1$, where $C_k>0$ is a constant.
So, 
$$\limsup_{t\rightarrow+\infty}\int_{K}|F_{l,t}|^2e^{-2\log|g_0|-2u_l+\psi}c(-\psi)<+\infty.$$
Not that $Y$ is discrete subset of $\Omega$. By Lemma \ref{l:converge} and the diagonal method,
there exists a subsequence of $\{F_{l,t}\}_{t\rightarrow+\infty}$ denoted by $\{F_{l,t_m}\}_{m\rightarrow+\infty}$
uniformly convergent to an  $F_{l}\in H^0(\Omega,\mathcal{O}(K_{\Omega}))$ on any compact subset of  $\Omega$. Then it follows from inequality \eqref{eq:211009c} and Fatou's Lemma that
\begin{equation}\label{eq:241203c}
\begin{split}
&\int_{\Omega}|F_{l}|^{2}e^{-2\log|g_0|-2u_l+\psi}c(-\psi)
\\=&\int_{\Omega}\liminf_{m\rightarrow+\infty}|F_{l,t_m}-(1-b_{t_m,1}(\psi))f|^{2}e^{-2\log|g_0|-2u_l+\psi}c(-\psi)\\
\leq&\liminf_{m\rightarrow+\infty}\int_{\Omega}|F_{l,t_m}-(1-b_{t_m,1}(\psi))f|^{2}e^{-2\log|g_0|-2u_l+\psi}c(-\psi)\\
\leq&\left(\int_0^{+\infty}c(s)e^{-s}ds\right)\sum_{1\le j\le m_1}2\pi\frac{|d_{1,j}|^2}{p_j|d_{2,j}|^2}e^{-2u_l(z_j)}.
\end{split}	
\end{equation}
Note that $\lim_{l\rightarrow+\infty}u_l(z_j)=u(z_j)>-\infty$ for  $1\le j\le m_1$.
By Lemma \ref{l:converge} ($g_l=e^{-2\log|g_0|-2u_l+\psi}c(-\psi)$),
there exists a subsequence of $\{F_{l}\}$ 
uniformly convergent to an  $ F\in H^0(\Omega,\mathcal{O}(K_{\Omega}))$  on any compact subset of $\Omega$, which satisfying ${F}=f+o(w_j^{k_{2,j}-1})dw_j$ near $z_j$  for any $j$. Taking a limit in inequality \eqref{eq:241203c}, we have inequality \eqref{eq:241203b} holds.
\end{proof}

\section{The concavity property of minimal $L^2$ integrals}

In this section, we recall some results about the concavity property of minimal $L^2$ integrals and prove the necessary conditions for the concavity degenerating to linearity on open Riemann surfaces.

\subsection{The concavity property on weakly pseudoconvex K\"ahler manifolds}
\label{sec:4.1}
\

In this section, we follow the notations $M$, $X$, $Z$, $\psi$, $\varphi$, $Z_0$, $f$, $\mathcal{F}$ and $\mathcal{P}_{T,M}$ in Theorem \ref{maintheorem}.

Denote that
\begin{displaymath}
	\begin{split}
		\mathcal{H}^2(c,t):=\bigg\{\tilde{f}\in H^0(\{\psi<-t\},\mathcal{O}(K_M)):&\int_{\{\psi<-t\}}|\tilde{f}|^2e^{-\varphi}c(-\psi)<+\infty\\
		&\&\,(\tilde{f}-f)\in H^0(Z_0,(\mathcal{O}(K_{M})\otimes\mathcal{F})|_{Z_0})\bigg\},
	\end{split}
\end{displaymath}
where $t\in[T,+\infty)$, $c$ is a nonnegative measurable function on $(T,+\infty)$.
The minimal $L^2$ integral is defined by 
$$G(t;c):=	\inf\{\int_{\{\psi<-t\}}|\tilde{f}|^{2}e^{-\varphi}c(-\psi):\tilde f\in \mathcal{H}^2(c,t)\}.$$
We denote $G(t;c)$ by $G(t)$ without misunderstanding.
Let $c\in\mathcal{P}_{T,M}$ satisfy that $\int_{T}^{+\infty}c(s)e^{-s}ds<+\infty$, and let $h(t):=\int_{t}^{+\infty}c(s)e^{-s}ds$. Assume $G(t)\not\equiv+\infty$, then Theorem \ref{maintheorem} shows that:

 \emph{$G(h^{-1}(r))$ is concave with respect to $r\in(0,\int_{T}^{+\infty}c(s)e^{-s}ds)$.}

The following lemma gives the existence and uniqueness of the minimal holomorphic section.
\begin{Lemma}[see \cite{GMY-concavity2}]\label{l:unique}
	 For any $t\ge T$, if $G(t)<+\infty$, there exists a unique  $F_t\in\mathcal{H}^2(c,t)$ such that $G(t)=\int_{\{\psi<-t\}}|F|^2e^{-\varphi}c(-\psi).$
\end{Lemma}	

If $G( {h}^{-1}(r))$ is linear, the unique minimal holomorphic sections on all $\{\psi<-t\}$
 is the same.
\begin{Proposition}[see \cite{GMY-concavity2}]
	\label{c:linear}
If $G( {h}^{-1}(r))$ is linear with respect to $r\in(0,\int_{T}^{+\infty}c(t)e^{-t}dt)$, 
	then there is a unique  $F\in \mathcal{H}^2(c,T)$ satisfying $G(t;c)=\int_{\{\psi<-t\}}|F|^2e^{-\varphi}c(-\psi)$ for all $t\geq T$ and
	\begin{equation}
		\nonumber
		\int_{M}|F|^2e^{-\varphi}a(-\psi)=\frac{G(T;c)}{\int_{T}^{+\infty}c(t)e^{-t}dt}\int_{T}^{+\infty} a(t)e^{-t}dt
	\end{equation}
	for any nonnegative measurable function $a$ on $(T,+\infty)$.
	Furthermore, if $\mathcal H^2(\tilde{c},t_0)\subset\mathcal H^2(c,t_0)$ for some $t_0\geq T$ and  nonnegative measurable function $\tilde{c}$ on $(T,+\infty)$,
	\begin{equation}
		\nonumber
		G(t_0;\tilde{c})=\int_{\{\psi<-t_0\}}|F|^2e^{-\varphi}\tilde{c}(-\psi)=\frac{G(T;c)}{\int_{T}^{+\infty}c(t)e^{-t}dt}\int_{t_0}^{+\infty} \tilde{c}(t)e^{-t}dt.\end{equation}	\end{Proposition}

We recall two remarks about the condition $\mathcal H^2(\tilde{c},t_0)\subset\mathcal H^2(c,t_0)$.

\begin{Remark}[see \cite{GY-concavity}]
	\label{r:c} Let $\tilde{c}\in\mathcal{P}_{T,M}$. If $\mathcal H^2(\tilde{c},t_1)\subset\mathcal H^2(c,t_1)$, then $\mathcal H^2(\tilde{c},t_2)\subset\mathcal H^2(c,t_2)$, where $t_1>t_2>T$. In the following, we give some sufficient conditions of
	$\mathcal H^2(\tilde{c},t_0)\subset\mathcal H^2(c,t_0)$ for $t_0> T$:
	
	$(1)$  $\lim_{t\rightarrow+\infty}\frac{\tilde{c}(t)}{c(t)}>0$;
	
	$(2)$  $\mathcal H^2(c,t_0)\not=\emptyset$ and there exists $t>t_0$, such that $\{\psi<-t\}\Subset\{\psi<-t_0\}$, $\{z\in\overline{\{\psi<-t\}}:\mathcal{I}(\varphi+\psi)_z\not=\mathcal{O}_z\}\subset Z_0$ and $\mathcal{F}|_{\overline{\{\psi<-t\}}}=\mathcal{I}(\varphi+\psi)|_{\overline{\{\psi<-t\}}}$.
\end{Remark}
In \cite{GY-concavity}, the above remark requires $M$ is Stein manifold and $c,\tilde c\in \mathcal{P}_{T}$, but its proof remains valid under the assumptions of the remark stated above.

 \begin{Remark}[see \cite{GY-concavity}]
	\label{l:c'}
	  If $c(t)$ is a positive measurable function on $(T,+\infty)$ such that $c(t)e^{-t}$ is  decreasing on $(T,+\infty)$ and $\int_{T}^{+\infty}c(t)e^{-t}dt<+\infty$, then there exists a positive measurable function $\tilde{c}$ on $(T,+\infty)$, satisfying the following statements:
	  
	  $(1)$ $\tilde{c}\geq c$ on $(T,+\infty)$ and $\int_{T}^{+\infty}\tilde{c}(t)e^{-t}dt<+\infty$;
	  
	  $(2)$ $\tilde{c}(t)e^{-t}$ is strictly decreasing on $(T,+\infty)$ and $\tilde{c}$ is  increasing on $(a,+\infty)$, where $a>T$ is a real number.
\end{Remark}

\subsection{Necessary conditions for  linearity to hold on open Riemann surfaces}
\label{sec:4.2}
\

Follow the notations in  section \ref{sec:4.1}. Let $M=\Omega$ be an open Riemann surface which admits a nontrivial Green function $G_{\Omega}$, and $X=Z=\emptyset$. Let $\psi<0$, i.e.. $T=0$.  

Let us recall a necessary condition of $G({h}^{-1}(r))$ is linear with respect to $r$, which will be used in the proof of Proposition \ref{p:infinite}.

\begin{Lemma}[see \cite{GY-concavity}]
\label{l:n}  
Let $c\in\mathcal{P}_{0,\Omega}$, and assume that $G(t)\not\equiv0$ or $+\infty$.
If $G( {h}^{-1}(r))$ is linear with respect to $r$, then there is no function $\tilde \varphi\geq\varphi$ ($\tilde\varphi\not=\varphi$) such that $\tilde\varphi+\psi\in SH(\Omega)$ and satisfies:
 	
$(1)$  $\mathcal{I}(\tilde\varphi+\psi)=\mathcal{I}(\varphi+\psi)$ and $\lim_{t\rightarrow 0+0}\sup_{\{\psi\geq-t\}}(\tilde\varphi-\varphi)=0$;
 	
$(2)$ There is an open set $U\Subset\Omega$ such that $\sup_{\Omega\backslash U}(\tilde\varphi-\varphi)<+\infty$, $\inf_Ue^{-\tilde\varphi}c(-\psi)>0$  and $\int_{U}|F_1-F_2|^2e^{-\varphi}c(-\psi)<+\infty$ for any $F_1,F_2\in H^0(\{\psi<-t\},\mathcal{O}(K_{\Omega}))$ satisfying $\int_{\{\psi<-t\}}|F_1|^2e^{-{\varphi}}{c}(-\psi)<+\infty$ 
and $\int_{\{\psi<-t\}}|F_2|^2e^{-\tilde{\varphi}}{c}(-\psi)<+\infty$ , where $t>0$ is any small enough number satisfying $U\Subset\{\psi<-t\}$. 	
\end{Lemma}

 When $\Omega=\Delta$ be the unit disc in $\mathbb{C}$, the following lemma holds by a simple calculation (a generalized result can be seen in \cite{GMY-concavity2}).
 
\begin{Remark}
\label{r:min} Let $\psi=2a\log|z|$ and $\varphi+\psi=2\log|g|+2(k+1)\log|z|+2u$, where $g\in\mathcal{O}(\Delta)$ satisfying $g(o)\not=0$ and $u$ is a harmonic function on $\Delta$. Let $f_{u}\in\mathcal{O}(\Delta)$ satisfy $|f_u(z)|=e^{u(z)}$. Then there is a constant $c_0\not=0$ such that  $c_0gf_uz^kdz$ is the unique holomorphic $(1,0)$ form $F$ on $\Delta$ satisfying $\int_{\Delta}|F|^2e^{-\varphi}c(-\psi)=\inf\{\int_{\Delta}|\tilde{F}|^2e^{-\varphi}c(-\psi):\tilde{F}=(z^k+o(z^k))dz$ near $o$ and $\tilde{F}\in H^0(\Delta,\mathcal{O}(K_{\Delta}))$$\}$.
\end{Remark}

 We give a lemma in real analysis, which will be used in the proof of Proposition \ref{p:infinite}.
 
 \begin{Lemma}
 	\label{l:c(t)e^{-at}}
 	Let $c(t)$ be a positive measurable function on $(0,+\infty)$, and let $a\in\mathbb{R}$. Assume that $\int_{t}^{+\infty}c(s)e^{-s}ds\in(0,+\infty)$ for any large enough $t$. Then 
 	
 	$(1)$ $\lim_{t\rightarrow+\infty}\frac{\int_{t}^{+\infty}c(s)e^{-as}ds}{\int_t^{+\infty}c(s)e^{-s}ds}=1$ if and only if $a=1$;
 	
 	$(2)$ $\lim_{t\rightarrow+\infty}\frac{\int_{t}^{+\infty}c(s)e^{-as}ds}{\int_t^{+\infty}c(s)e^{-s}ds}=0$ if and only if $a>1$;
 	
 	$(3)$ $\lim_{t\rightarrow+\infty}\frac{\int_{t}^{+\infty}c(s)e^{-as}ds}{\int_t^{+\infty}c(s)e^{-s}ds}=+\infty$ if and only if $a<1$.
 \end{Lemma}
 \begin{proof}
 	If $a=1$, it clear that $\lim_{t\rightarrow+\infty}\frac{\int_{t}^{+\infty}c(s)e^{-as}ds}{\int_t^{+\infty}c(s)e^{-s}ds}=1$.
 	
 	If $a>1$, then $c(s)e^{-as}\le e^{(1-a)s_0} c(s)e^{-s}$ for $s\ge s_0>0$, which implies that $\limsup\limits_{t\rightarrow+\infty}\frac{\int_{t}^{+\infty}c(s)e^{-as}ds}{\int_t^{+\infty}c(s)e^{-s}ds}\le e^{(1-a)s_0}$. Let $s_0\rightarrow+\infty$, we have $\lim\limits_{t\rightarrow+\infty}\frac{\int_{t}^{+\infty}c(s)e^{-as}ds}{\int_t^{+\infty}c(s)e^{-s}ds}=0.$
 	
 	If $a<1$, then $c(s)e^{-as}\ge e^{(1-a)s_0} c(s)e^{-s}$ for $a>s_0>0$, which implies that $\liminf\limits_{t\rightarrow+\infty}\frac{\int_{t}^{+\infty}c(s)e^{-as}ds}{\int_t^{+\infty}c(s)e^{-s}ds}\ge e^{(1-a)s_0}$. Letting $s_0\rightarrow+\infty$, $\lim\limits_{t\rightarrow+\infty}\frac{\int_{t}^{+\infty}c(s)e^{-as}ds}{\int_t^{+\infty}c(s)e^{-s}ds}=+\infty.$
 \end{proof}

 Let $Z_0'$, $z_j$, $w_j$, $V_{z_j}$, $V_0$, $\varphi$ and $\psi$ be as in Proposition \ref{p:extension}.
 Let $c(t)\in\mathcal{P}_{0,\Omega}$ satisfy $\int_{0}^{+\infty}c(s)e^{-s}ds<+\infty$, and let $\mathcal{F}_{z_j}\supset\mathcal{I}(\varphi+\psi)_{z_j}$ be an ideal of $\mathcal{O}_{z_j}$ for any $1\le j<\gamma$. Let $f\in H^0(V_0,\mathcal{O}(K_{\Omega}))$ be the $(1,0)$ form  in the definition of $G(t)$.
 
Now, we prove a necessary condition for $G(h^{-1}(r))$ is linear.

\begin{Proposition}
	\label{p:infinite}
	Assume that $G(0)\in(0,+\infty)$ and $(\psi-2p_jG_{\Omega}(\cdot,z_j))(z_j)>-\infty$ for any $j$, where $p_j=\frac{1}{2}\nu(dd^c\psi,z_j)>0$. If $G(h^{-1}(r))$ is linear with respect to $r$, then 
	
	$(1)$ $\psi=2\sum_{1\le j<\gamma}p_jG_{\Omega}(\cdot,z_j)$;

	$(2)$ $\varphi+\psi=2\log|g|$ and $\mathcal{F}_{z_j}=\mathcal{I}(\varphi+\psi)_{z_j}$ for any $j$, where $g\in\mathcal{O}(\Omega)$ such that $ord_{z_j}(g)=ord_{z_j}(f)+1$ for any $j$;
	
	$(3)$ There exists $c_0\in\mathbb{C}\backslash\{0\}$ such that $\frac{p_j}{ord_{z_j}g}\lim_{z\rightarrow z_j}\frac{dg}{f}=c_0$ for any $j$;
	
	$(4)$ $\sum_{1\le j<\gamma}p_j<+\infty$.
\end{Proposition}
\begin{proof}
	As $\varphi+\psi$ is subharmonic on $\Omega$,  we have
	\begin{equation}
		\label{eq:211007a}
		\varphi+\psi=2\log|g_0|+2u,
	\end{equation}
	where $g\in\mathcal{O}(\Omega)$ and $u\in SH(\Omega)$  such that $\nu(dd^cu,z)\in[0,1)$ for any $z\in\Omega$.
	It follows from  Siu's Decomposition Theorem and Lemma \ref{l:green-sup2} that
	\begin{equation}
		\label{eq:211007b}
		\psi=2\sum_{1\le j<\gamma}p_jG_{\Omega}(\cdot,z_j)+\psi_2,
	\end{equation}
	where $\psi_2\in SH^-(\Omega)$ satisfying $\psi(z_j)>-\infty$ for any $j$.
	
	We prove the proposition in four steps. 

\
	
	\emph{\textbf{Step 1.} $\psi=2\sum_{1\le j<\gamma}p_jG_{\Omega}(\cdot,z_j)$}

	Since $\Omega$ is a Stein manifold and $u,\psi_2\in SH(\Omega)$,  there are  $u_l,\Psi_l\in SH(\Omega)\cap C^{\infty}(\Omega)$  decreasingly convergent to $u$  and  $\psi_2$ with respect to $l$ (see \cite{FN80}), respectively.
	
	As $G(h^{-1}(r))$ is linear,  Proposition \ref{c:linear} shows that there exists an $F\in H^0(\Omega,\mathcal{O}(K_{\Omega}))$ such that $(F-f,z_j)\in(\mathcal{O}(K_{\Omega}))_{z_j}\otimes\mathcal{F}_{z_j}$ for any $1\le j<\gamma$ and 	for any $t\ge0$,
	\begin{equation}
		\nonumber
		G(t)=\left(\int_{t}^{+\infty}c(s)e^{-s}ds\right)\int_{\{\psi<-t\}}|F|^2e^{-\varphi}c(-\psi)
	\end{equation}
 By Remark \ref{l:c'} and Proposition \ref{c:linear},  assume that $c$ is increasing near $+\infty$ without loss of generality.
	
	Denote $\mathcal{G}:=2\sum_{1\le j<\gamma}p_jG_{\Omega}(\cdot,z_j)$. Combining $\psi_2\le\Psi_l$, $u\le u_l$, $c$ is increasing near $+\infty$, equality \eqref{eq:211007a} and equality \eqref{eq:211007b}, we obtain that there exists $t_1>0$ such that for any $t>t_1$,
	\begin{equation}
		\label{eq:211007d}\begin{split}
			\int_{\{\psi<-t\}}|F|^2e^{-\varphi}c(-\psi)
			=&\int_{\{\mathcal{G}+\psi_2<-t\}}	|F|^2e^{-2\log|g_0|-2u+\mathcal{G}+\psi_2}c(-\mathcal{G}-\psi_2)	\\
			\ge&\int_{\{\mathcal{G}+\Psi_l<-t\}}	|F|^2e^{-2\log|g_0|-2u_l+\mathcal{G}+\psi_2}c(-\mathcal{G}-\Psi_l).	
	\end{split}\end{equation}

	For any $\epsilon>0$ and any $m$ $(1\le m<\gamma+1)$, there exists $s_0>0$ satisfying that:
	
	$(1)$ $\{|w_j(z)|<s_0:z\in V_{z_j}\}\cap\{|w_k(z)|<s_0:z\in V_{z_k}\}=\emptyset$ for any $j\not=k$, Denote that $U_j:=\{|w_j(z)|<s_0:z\in V_{z_j}\}$;
	
	$(2)$  $\sup_{z\in U_j}|u_l(z)-u_l(z_j)|<\epsilon$ for any $j\in\{1,2,\ldots,m\}$;
	
	$(3)$ $\sup_{z\in U_j}|h_j(z)-h_j(z_j)|<\epsilon$ for any $j\in\{1,2,\ldots,m\}$, where $h_j:=\mathcal{G}-2p_j\log|w_j|+\Psi_l+\epsilon$ are  smooth functions on $U_j$;
	
	$(4)$ there exists a  $\tilde{g}_j\in\mathcal{O}(\cup_{1\le j\le m}U_j)$ such that  $|\tilde{g}_j|^2=e^{\frac{\mathcal{G}}{p_j}}$.
	
	Note that $\mathcal{G}+\Psi_l\le 2p_j\log|w_j|+h_j(z_j)$ on $U_j$ and $ord_{z_j}\tilde{g}_j=1$ for any $j$. It follows from Lemma \ref{l:G-compact} that there exists $t_2>t_1$ such that $\{\mathcal{G}+\Psi_l<-t_2\}\cap (\cup_{1\le j\le m}V_{z_j}) \Subset \cup_{1\le j\le m} U_j$.  Then inequality \eqref{eq:211007d} becomes that
	\begin{equation}
		\label{eq:211007f}
		\begin{split}
			\int_{\{\psi<-t\}}|F|^2e^{-\varphi}c(-\psi)
			\ge& \sum_{1\le j\le m}\int_{\{\mathcal{G}+\Psi_l<-t\}\cap U_j}	|F|^2e^{-2\log|g_0|-2u_l+\mathcal{G}+\psi_2}c(-\mathcal{G}-\Psi_l)\\
			\ge& \sum_{1\le j\le m}\int_{\{2p_j\log|w_j|+h_j(z_j)<-t\}\cap U_j}	|F|^2|\tilde{g}_j|^{2p_j}\\
			&\cdot e^{-2\log|g_0|-2u_l(z_j)-\epsilon+\psi_2}c(-2p_j\log|w_j|-h_j(z_j)).
		\end{split}
	\end{equation}
	Let $F=d_{1,j}w_j^{k_{1,j}}h_{1,j}dw_{j}$, $g_0=d_{2,j}w_{j}^{k_{2,j}}h_{2,j}$ and $\tilde{g}_j=d_{3,j}w_{j}h_{3,j}$ on $U_j$, where $d_{i,j}\not=0$ are constants, $k_{1,j},k_{2,j}\in\mathbb{Z}_{\ge0}$ , and $h_{i,j}\in\mathcal{O}(U_j)$ such that $h_{i,j}(z_j)=1$ for $i\in\{1,2,3\}$ and $j\in\{1,2,\ldots,m\}$. Then inequality \eqref{eq:211007f} implies
	\begin{equation}
		\label{eq:211024a}\begin{split}
			&\int_{\{\psi<-t\}}|F|^2e^{-\varphi}c(-\psi)\\
			\ge
			&\sum_{1\le j\le m}\bigg(\frac{2|d_{1,j}|^2|d_{3,j}|^{2p_j}e^{-2u_l(z_j)-\epsilon}}{|d_{2,j}|^2}\\
			&\times\int_{0}^{2\pi}\int_{0}^{e^{-\frac{t+h_{j}(z_j)}{2p_j}}}|r|^{2(k_{1,j}+p_j-k_{2,j})+1}|\frac{h_{1,j}}{h_{2,j}}|^2|h_{3,j}|^{2p_j}e^{\psi_2}c(-2p_j\log r-h_{j}(z_j))drd\theta\bigg)\\
			\ge&\sum_{1\le j\le m}\bigg(|\frac{d_{1,j}}{d_{2,j}}|^2|d_{3,j}|^{2p_j}\frac{2\pi}{p_j} e^{-2u_l(z_j)-\epsilon+\psi_{2}(z_j)-\frac{k_{1,j}+1-k_{2,j}+p_j}{p_j}h_{j}(z_j)}\\
			&\times\int_{t}^{+\infty}c(s)e^{-(\frac{k_{1,j}+1-k_{2,j}}{p_j}+1)s}ds\bigg).
		\end{split}
	\end{equation}
	Denote that $I_{0}:=\{1\le j<\gamma:ord_{z_j}F+1-ord_{z_j}g_0\le0\}$ and $I_m:=\{j\le m:j\in I_0\}$.  Note that for any $t\ge0$, $$\frac{\int_{\{\psi<-t\}}|F|^2e^{-\varphi}c(-\psi)}{\int_t^{+\infty}c(s)e^{-s}ds}=\frac{G(t)}{\int_t^{+\infty}c(s)e^{-s}ds}=\frac{G(0)}{\int_0^{+\infty}c(s)e^{-s}ds}\in(0,+\infty)$$ 
	By Lemma \ref{l:c(t)e^{-at}} and inequality \eqref{eq:211024a}, we have
	\begin{equation}
		\label{eq:241203d} k_{1,j}+1-k_{2,j}=0
	\end{equation}
 for any $j\in I_m$  and
	\begin{equation}\nonumber
		\begin{split}
			&\frac{\int_{\{\psi<-t\}}|F|^2e^{-\varphi}c(-\psi)}{\int_t^{+\infty}c(s)e^{-s}ds}\\
			\ge&\sum_{1\le j\le m}\bigg(|\frac{d_{1,j}}{d_{2,j}}|^2|d_{3,j}|^{2p_j}\frac{2\pi}{p_j} e^{-2u_l(z_j)-\epsilon+\psi_{2}(z_j)-\frac{k_{1,j}+1-k_{2,j}+p_j}{p_j}h_{j}(z_j)}\\
			&\times\lim_{t\rightarrow+\infty}\frac{\int_{t}^{+\infty}c(s)e^{-(\frac{k_{1,j}+1-k_{2,j}}{p_j}+1)s}ds}{\int_t^{+\infty}c(s)e^{-s}ds}\bigg)\\
			=&\sum_{j\in I_m}|\frac{d_{1,j}}{d_{2,j}}|^2|d_{3,j}|^{2p_j}\frac{2\pi}{p_j} e^{-2u_l(z_j)-\epsilon+\psi_{2}(z_j)-h_{j}(z_j)}.
		\end{split}
	\end{equation}
	Since $h_j(z_j)=\Psi_l(z_j)+\epsilon+\lim_{z\rightarrow z_j}(\mathcal{G}-2p_j\log|w_j|)=\Psi_l(z_j)+\epsilon+\log(\lim_{z\rightarrow z_j}\frac{|\tilde{g}|^{2p_j}}{|w_j|^{2p_j}})=\Psi_l(z_j)+\epsilon+2p_j\log|d_{3,j}|$, letting $\epsilon\rightarrow0$, we have 
	\begin{equation}
		\nonumber
		\begin{split}
			\frac{\int_{\{\psi<-t\}}|F|^2e^{-\varphi}c(-\psi)}{\int_t^{+\infty}c(s)e^{-s}ds}
			\ge\sum_{j\in I_m}\frac{|d_{1,j}|^2}{p_j|d_{2,j}|^2}2\pi e^{-2u_l(z_j)+\psi_{2}(z_j)-\Psi_l(z_j)}.
		\end{split}
	\end{equation}
	Letting $l\rightarrow+\infty$ and $m\rightarrow\gamma+1$, we have
	\begin{equation}
		\label{eq:211101a}
		\frac{G(0)}{\int_0^{+\infty}c(s)e^{-s}ds}=\frac{\int_{\{\psi<-t\}}|F|^2e^{-\varphi}c(-\psi)}{\int_t^{+\infty}c(s)e^{-s}ds}
		\ge\sum_{j\in I_0}\frac{|d_{1,j}|^2}{p_j|d_{2,j}|^2}2\pi e^{-2u(z_j)}.
	\end{equation}
	Then we have $u(z_j)>-\infty$  for $j\in I_0$.
	Note that $(w_j^{k_{2,j}})_o\in\mathcal{I}(\varphi+\psi)_{z_j}\subset\mathcal{F}_{z_j}$ for any $j$. It follows from Proposition \ref{p:extension} that
	\begin{equation}
		\label{eq:211009f}G(0)\le\left(\int_0^{+\infty}c(s)e^{-s}ds\right)\sum_{j\in I_0}\frac{|d_{1,j}|^2}{p_j|d_{2,j}|^2}2\pi e^{-2u(z_j)}.
	\end{equation}
	Combining inequality \eqref{eq:211101a} and \eqref{eq:211009f}, we have
	\begin{equation}
		\label{eq:211009g}
		G(0)=\int_{\Omega}|F|^2e^{-\varphi}c(-\psi)=\left(\int_0^{+\infty}c(s)e^{-s}ds\right)\sum_{j\in I_0}\frac{|d_{1,j}|^2}{p_j|d_{2,j}|^2}2\pi e^{-2u(z_j)}.
	\end{equation}
	
	Denote $\tilde{\psi}:=2\sum_{1\le j<\gamma}p_jG_{\Omega}(\cdot,z_j)$ and $\tilde{\varphi}:=\varphi+\psi-\tilde\psi$. As $c(t)e^{-t}$ is decreasing and $\tilde{\psi}\ge\psi$, by Proposition \ref{p:extension}, there exists  $\tilde{F}\in H^0(\Omega,\mathcal{O}(K_{\Omega}))$ such that $\tilde{F}=F+o(w_j^{k_{2,j}-1})dw_j$ near $z_j$ for any $j$ and
	\begin{equation}
		\label{eq:211009i}\begin{split}
			\int_{\Omega}|\tilde{F}|^2e^{-{\varphi}}c(-{\psi})
			\le&\int_{\Omega}|\tilde{F}|^2e^{-\tilde{\varphi}}c(-\tilde{\psi})\\
			\le&\left(\int_0^{+\infty}c(s)e^{-s}ds\right)\sum_{j\in I_0}\frac{|d_{1,j}|^2}{p_j|d_{2,j}|^2}2\pi e^{-2u(z_j)}.		
	\end{split}\end{equation}
	Following from $\int_{\Omega}|\tilde{F}|^2e^{-{\varphi}}c(-{\psi})\ge G(0)$,    equality \eqref{eq:211009g} and \eqref{eq:211009i}, we have
	\begin{displaymath}
		\int_{\Omega}|\tilde{F}|^2e^{-{\varphi}}c(-{\psi})
		=G(0)=\int_{\Omega}|\tilde{F}|^2e^{-\tilde{\varphi}}c(-\tilde{\psi})>0.
	\end{displaymath}
Using Lemma \ref{l:psi=G}, we have $\psi=2\sum_{1\le j<\gamma}p_jG_{\Omega}(\cdot,z_j)$.

\	
	
	\emph{\textbf{Step 2.} $u$ is harmonic, $\mathcal{F}_{z_j}=\mathcal{I}(\varphi+\psi)_{z_j}$ and $ord_{z_j}g_0=ord_{z_j}f+1$ for any $j$.}

	Following equality \eqref{eq:241203d},  we have $ord_{z_j}F-ord_{z_j}g_0\ge-1$ for any $1\le j<\gamma$.
	We prove $ord_{z_j}F-ord_{z_j}g_0=-1$ and $\mathcal{F}_{z_j}=\mathcal{I}(\varphi+\psi)_{z_j}$ for any $j$ by contradiction: if not, there exists a $j_0$ such that $ord_{z_{j_0}}F-ord_{z_{j_0}}g_0\ge0$ or $\mathcal{I}(\varphi+\psi)_{z_{j_0}}\subsetneqq\mathcal{F}_{z_{j_0}}$, then $$(0-F,z_{j_0})\in(\mathcal{O}(K_{\Omega}))_{z_{j_0}}\otimes\mathcal{F}_{z_{j_0}}.$$ There exists $r_0>0$ such that $U_0=\{|w_{j_0}(z)|<r_0:z\in V_{j_0}\}\Subset V_{j_0}$.  As $\psi=2\sum_{1\le j<\gamma}G_{\Omega}(\cdot,z_j)$, there exists $s_0>0$ such that $\{\psi<-s_0\}\cap \partial U_0=\emptyset$. Let $\tilde{F}:=\left\{ \begin{array}{lcl}
		F & \mbox{on}& \Omega\backslash\overline{U_0}\\
		0 & \mbox{on}& U_0
	\end{array} \right.$ be a holomorphic $(1,0)$ form on $\Omega\backslash\partial U_0$, which satisfies that $(\tilde{F}-F,z_{j})\in(\mathcal{O}(K_{\Omega}))_{z_{j}}\otimes\mathcal{F}_{z_{j}}$ for any $j$.
	Then we have $G(s_0)\le\int_{\{\psi<-s_0\}}|\tilde{F}|^2e^{-\varphi}c(-\psi)\le\int_{\{\psi<-s_0\}}|{F}|^2e^{-\varphi}c(-\psi)=G(s_0)$, which implies that
	\begin{equation}
		\label{eq:211010a}
		\int_{\{\psi<-s_0\}}|\tilde{F}|^2e^{-\varphi}c(-\psi)=\int_{\{\psi<-s_0\}}|{F}|^2e^{-\varphi}c(-\psi).
	\end{equation}
	As $F\in H^0(\Omega,\mathcal{O}(K_{\Omega}))$,  equality \eqref{eq:211010a} shows that $F\equiv0$, which contradicts to $G(0)=\int_{\Omega}|F|^2e^{-\varphi}c(-\psi)>0$. Thus, we have $ord_{z_j}f=ord_{z_j}F=ord_{z_j}g_0-1$ and $\mathcal{F}_{z_j}=\mathcal{I}(\varphi+\psi)_{z_j}$ for any $j$ by $(f-F,z_{j})\in(\mathcal{O}(K_{\Omega}))_{z_{j}}\otimes\mathcal{F}_{z_{j}}=(\mathcal{O}(K_{\Omega})\otimes\mathcal{I}(2\log|g_0|))_{z_{j}}$, which implies that $I_0=\{j\in\mathbb{Z}:1\le j<\gamma\}$.

	Now, we assume that $u$ is not harmonic to get a contradiction. There exists $p\in{\Omega}$ such that $u$ is not harmonic on any neighborhood of $p$.
	
	If $p\in\Omega\backslash Z_0'$, let $U$ be an open subset of $\Omega$ and $t_0>t_1>0$ such that $p\in U\Subset \{\psi<-t_1\}\backslash \overline{\{\psi<-t_0\}}$. Then there exists a closed  positive $(1,1)$ current $T\not\equiv0$, such that $supp T\Subset U$ and $T\leq i\partial\bar\partial u$.
	By Lemma \ref{l:cu}, there exists $\Phi\in SH(\Omega)$, which satisfies the following properties: $i\partial\bar\partial\Phi\leq T$ and $i\partial\bar\partial\Phi\not\equiv0$; $\lim_{t\rightarrow0+0}(\inf_{\{G_{\Omega}(z,z_1)\geq-t\}}\Phi(z))=0$, which implies that $\lim_{t\rightarrow0+0}(\inf_{\{\psi\geq-t\}}\Phi(z))=0$; $supp (i\partial\bar\partial\Phi)\subset U$ and $\inf_{\Omega\backslash U}\Phi>-\infty$.
	
	Take $\tilde\varphi=\varphi-\Phi$, then $\tilde\varphi=2\log|g_0|-2\sum_{1\le j<\gamma}p_jG_{\Omega}(\cdot,z_j)+2u-\Phi$ is subharmonic on a neighborhood of $\overline{U}$. It is clear that $\tilde\varphi\geq\varphi$, $\inf_{\Omega\backslash U}(\varphi-\tilde\varphi)=\inf_{\Omega\backslash U}\Phi>-\infty$, $\tilde\varphi+\psi\in SH(\Omega)$, $\mathcal{I}(\tilde\varphi+\psi)=\mathcal{I}(\varphi+\psi)=\mathcal{I}(2\log|g_0|)$.
	As $\tilde\varphi\in SH(\Omega)$ and $\inf_{\overline{U}}>0$. Note that for any $z\in \overline{U}$, $$\mathcal{I}(\varphi)_z=\mathcal{I}(2\log|g_0|)_z=\mathcal{I}(\tilde{\varphi})_z\text{ and } e^{-\tilde{\varphi}}c(-\psi)(z)\ge(\inf_{t\in[t_1,t_0]}c(t))e^{-\tilde{\varphi}}(z).$$
	 For any  $F_1,F_2\in H^0(\{\psi<-t\},\mathcal{O}(K_{\Omega}))$ satisfying $\int_{\{\psi<-t\}}|F_1|^2e^{-\tilde{\varphi}}c(-\psi)<+\infty$ and $\int_{\{\psi<-t\}}|F_2|^2e^{-\varphi}c(-\psi)<+\infty$, where $U\Subset\{\psi<-t\}$,
	 we have
	$$(F_i,z)\in\mathcal({O}(K_{\Omega}))_z\otimes\mathcal{I}(\tilde{\varphi})_z=({O}(K_{\Omega}))_z\otimes\mathcal{I}({\varphi})_z$$
	for any $z\in\overline{U}$ and $i=1,2$, which implies that
	\begin{equation*}
			\int_{U}|F_1-F_2|^2e^{-\varphi}c(-\psi)
			\le(\sup_{t\in[t_1,t_0]}c(t))\int_{U}|F_1-F_2|^2e^{-\varphi}<+\infty.
	\end{equation*}
Then $\tilde\varphi$ satisfies all conditions in Lemma \ref{l:n}, which is a contradiction.
	
	If $p\in Z_0'$ (without loss of generality, we can assume $p=z_1$), there exists $s_1>0$ such that $\{\psi<-s_1\}\cap\partial U_1=\emptyset$, where $U_1=\{|w_1(z)|<r_1:z\in V_{z_1}\}\Subset V_{z_1}$ is a neighborhood of $z_1$. For $l=1,2$, denote that
	\begin{equation*}
		\begin{split}
			G_l(t)= \inf\bigg\{\int_{\{\psi<-t\}\cap D_l}|\tilde{f}|^{2}e^{-\varphi}&c(-\psi):\tilde{f}\in H^{0}(\{\psi<-t\}\cap D_l,\mathcal{O}(K_{\Omega}))  \\&\&{\,}(\tilde{f}-f,z_j)\in(\mathcal{O}(K_{\Omega}))_{z_j}\otimes\mathcal{F}_{z_j}\,for\,any\,z_j\in D_l\bigg\},
		\end{split}
	\end{equation*}
	where $t\ge s_1$, $D_1=\{\psi<-s_1\}\cap U_1$ and $D_2=\{\psi<-s_1\}\backslash\overline{U_1}$. Theorem \ref{maintheorem} shows that $G_l(h^{-1}(r))$  is concave on $(0,\int_{s_1}^{+\infty}c(s)e^{-s}ds)$. Since $G(t)=G_1(t)+G_2(t)$  and $G(h^{-1}(r))$ is linear, we have $G_l(h^{-1}(r))$ is linear on $(0,\int_{s_1}^{+\infty}c(s)e^{-s}ds)$ for $l=1,2$. 
	Replacing $c$ by $1$ in the definition of $G_1(t)$, we define  minimal $L^2$ integrals $\tilde{G}_1(t)$ for $t\ge s_1$.
	Note that $\frac{1}{2p_1}(\psi+s_1)$ is the Green function $G_{D_1}(\cdot,z_1)$ on $D_1$. Combining Proposition \ref{c:linear}, Remark \ref{r:c} and the linearity of $G_1(h^{-1}(r))$, we have $\tilde{G}_1(-\log r)$ is linear on $(0,e^{-s_1})$. 
	
	Similarly, 
there exist an open subset $U\Subset D_1$ ($p\in U$) and $\Phi\in SH^-(D_1)$  such that  that $\lim_{t\rightarrow s_1+0}(\inf_{\{z\in D_1:\psi(z)\geq-t\}}\Phi(z))=0$, $0\not=i\partial\bar\partial\Phi\le i\partial\bar\partial u$ and $supp (i\partial\bar\partial\Phi)\subset U$ and $\inf_{D_1\backslash U}\Phi>-\infty$.
	Without loss of generality, we can assume that $p_1=1$ by the following remark. 
\begin{Remark}
	\label{r:c'}
	Let $\tilde{\varphi}=\varphi+a\psi$, $\tilde{c}(t)=c(\frac{t}{1-a})e^{-\frac{at}{1-a}}$ and $\tilde{\psi}=(1-a)\psi$ for some $a\in(-\infty,1)$. It is clear that $e^{-\tilde{\varphi}}c(-\tilde{\psi})=e^{-\varphi}c(-\psi)$, $(1-a)\int_{t}^{+\infty}c(s)e^{-s}ds=\int_{(1-a)t}^{+\infty}\tilde{c}(s)e^{-s}ds$ and $G(t;\varphi,\psi,c)=G((1-a)t;\tilde{\varphi},\tilde{\psi},\tilde{c})$.
\end{Remark}

	Take 
	$\tilde\varphi=\varphi-\Phi,$
	then $\tilde\varphi=2\log|g_0|-2\sum_{1\le j<\gamma}p_jG_{\Omega}(\cdot,z_j)+2u-\Phi\in SH(D_1)$. It is clear that $\tilde\varphi\geq\varphi$, $\inf_{D_1\backslash U}(\varphi-\tilde\varphi)=\inf_{D_1\backslash U}\Phi>-\infty$, $\tilde\varphi+\psi\in SH(D_1)$, $\inf_{\overline{U}}e^{-\tilde\varphi}>0$ and $\mathcal{I}(\tilde\varphi+\psi)_z=\mathcal{I}(\varphi+\psi)_z=\mathcal{I}(2\log|g_0|)_z$ for any $z\in D_1$.
 Note that $\mathcal{I}(\varphi)|_z=\mathcal{I}(\tilde\varphi)|_{z}=\mathcal{I}(2\log|g_0|-2G_{\Omega}(\cdot,z_j))|_z$ for any $z\in\overline{U}$, then
	$$\int_{U}|F_1-F_2|^2e^{-\varphi}<+\infty$$
	for any  $F_1,F_2\in H^0(\{\psi<-t\},\mathcal{O}(K_{\Omega}))$ satisfying $\int_{\{\psi<-t\}\cap D_1}|F_1|^2e^{-\tilde{\varphi}}<+\infty$ and $\int_{\{\psi<-t\}\cap D_1}|F_2|^2e^{-\varphi}<+\infty$, where $U\Subset\{\psi<-t\}\cap D_1$. Then $\tilde\varphi$ satisfies all conditions in Lemma \ref{l:n}, which  contradicts to that $\tilde{G}_1(-\log r)$ is  linear on $(0,\int_{s_1}^{+\infty}c(s)e^{-s}ds)$.
	
	Thus, $u$ is harmonic on $\Omega$.

	\

	\emph{\textbf{Step 3.} $\varphi+\psi=2\log|g|$ and $\lim_{z\rightarrow z_j}\frac{f}{dg}=c_0\in\mathbb{C}\backslash{0}$ for any $j$.}

	We follow the notations $D_l$,  $s_1$ and $G_l$ in Step 2, where $l=1,2$. Then $G(t)=G_1(t)+G_2(t)$  and  $G_l(h^{-1}(r))$ is linear on $(0,\int_{s_1}^{+\infty}c(s)e^{-s}ds)$ for $l=1,2$. 
	 Note that $D_1$ is  simply connected and $\frac{\psi+s_1}{2p_{1}}|_{D_1}$ is the Green function $G_{D_1}(\cdot,z_1)$ on $D_1$.
 There exist  $H_1,H_2\in\mathcal{O}(D_1)$  such that $|H_1|=e^{\frac{\psi+s_1}{2p_{1}}}$ and $|H_2|=e^{u}$ on $D_1$.
	Then we have
	\begin{displaymath}
		\begin{split}
			\varphi+\psi=2\log\left|\frac{g_0}{H_1^{ord_{z_1}f+1}}\right|+2(ord_{z_1}f+1)\frac{\psi+s_1}{2p_{1}}+2u
		\end{split}
	\end{displaymath}
	on $D_1$, where  $\frac{g_0}{H_1^{ord_{z_1}f+1}}$ is holomorphic  on $D_1$.
	
	Let $\tilde{p}:\Delta\rightarrow\Omega\backslash\{z_j:1<j<\gamma\}$ be the universal covering from unit disc $\Delta$ to $\Omega\backslash Z_0'$, and let $\tilde{D}_1$ be an open subset of $\Delta$ such that $\tilde{p}|_{\tilde{D}_1}$ is a conformal map from $\tilde{D}_1$ to $D_1$.
	Thus, Remark \ref{r:min} ($g\sim \frac{g_0}{H_1^{ord_{z_1}f+1}}$ and $\log|z|\sim \frac{\psi+s_1}{2p_{1}}$, here `$\sim$' means that the former is replaced by the latter) shows that
	\begin{equation}\label{eq:211024b}
		\begin{split}
			F=&\tilde{c}_0\frac{g_0}{H_1^{ord_{z_1}f+1}}H_2H_1^{ord_{z_j}f}dH_1
			=\tilde{c}_0g_0\frac{dH_1}{H_1}H_2\\
			=&\tilde{c}_1g_0(\tilde{p}|_{\tilde{D}_1})_*(\frac{d\tilde{H}_1}{\tilde{H}_1}\tilde{H}_2)
			=\tilde{c}_1g_0(\tilde{p}|_{\tilde{D}_1})_*(\frac{d\tilde{H}_1}{\tilde{H}_1})(\tilde{p}|_{\tilde{D}_1})_*(\tilde{H}_2)
		\end{split}
	\end{equation}
	on $D_1$, where  $\tilde{c}_{l}\not=0$ is a constant for $l=0,1$, and $\tilde{H}_1, \tilde{H}_2\in\mathcal{O}(\Delta)$  satisfying $|\tilde{H}_1|=\tilde{p}^*e^{\frac{\psi+s_1}{2p_{1}}}$ ($\frac{\psi+s_1}{2p_{1}}$ is harmonic on $\Omega\backslash Z_0'$) and $|\tilde{H}_2|=\tilde{p}^*e^{u}$. Equality \eqref{eq:211024b} shows that
	\begin{equation}
		\label{eq:211101d}\tilde{p}^*({F})=\tilde{c}_1\tilde{p}^*(g_0)\frac{d\tilde{H}_1}{\tilde{H}_1}\tilde{H}_2
	\end{equation}
	on $\Delta$.
	As $g_0\tilde{p}_*(\frac{d\tilde{H}_1}{\tilde{H}_1})$ is a (single-value) holomorphic $(1,0)$ form on $\Omega\backslash \{z_j:1<j<\gamma\}$, it follows from equality \eqref{eq:211101d} that $\tilde{p}_*(\tilde{H}_2)$ is a (single-value) holomorphic function on $\Omega\backslash \{z_j:1<j<\gamma\}$ satisfying $|\tilde{p}_*(\tilde{H}_2)|=e^{u}$. Thus, 
	there exists a  $g_1\in\mathcal{O}(\Omega)$ such that $|g_1|=e^{u}$. Let $g=g_0g_1$, then we have $\varphi+\psi=2\log|g|$. $ord_{z_j}g_0=ord_{z_j}f+1$ shows that $ord_{z_j}g=ord_{z_j}f+1$ for any $1\le j<\gamma$.
	
	Fixed $j$ $(1<j<\gamma)$, there exists a simple connected open subset $D_3$ of $\Omega$ such that $z_1,z_j\in D_3$ and $z_k\not\in D_3$ for any $k\not=1,j$. There exist   $\tilde{f}_{z_1},\tilde{f}_{z_j},H_3\in\mathcal{O}(D_3)$  satisfying $|\tilde{f}_{z_1}|=e^{G_{\Omega}(\cdot,z_1)}$,  $|\tilde{f}_{z_j}|=e^{G_{\Omega}(\cdot,z_j)}$ and  $|H_3|=e^{\frac{1}{p_1}\sum_{k\not=1,j}p_kG_{\Omega}(\cdot,z_k)}$. Without loss of generality, we can assume $D_1\subset D_3$. Then we have
	\begin{equation}
		\label{eq:211102a}\begin{split}
			F=\tilde{c}_2g\frac{d{H}_1}{{H}_1}=\tilde{c}_3g\frac{d(\tilde{f}_{z_1}\tilde{f}_{z_j}^{\frac{p_j}{p_1}}H_3)}{\tilde{f}_{z_1}\tilde{f}_{z_j}^{\frac{p_j}{p_1}}H_3}=\tilde{c}_3g\left(\frac{d(\tilde{f}_{z_1}H_3)}{\tilde{f}_{z_1}H_3}+\frac{p_jd\tilde{f}_{z_j}}{p_1\tilde{f}_{z_j}}\right)
	\end{split}\end{equation}
	on $D_1$, where $\tilde{c}_2\not=0$ and $\tilde{c}_3\not=0$ are constants. As $\tilde{f}_{z_1},\tilde{f}_{z_j},H_3\in\mathcal{O}(D_3)$,  equality \eqref{eq:211102a} holds on $D_3$.
	
	As $(F-f,z_k)\in(\mathcal{O}(K_{\Omega})\otimes\mathcal{I}(\varphi+\psi))_{z_k}$, $\varphi+\psi=2\log|g|$ and $ord_{z_k}g=ord_{z_k}f+1$ for any $1\le k<\gamma$,  we have
		$\lim_{z\rightarrow z_1}\frac{F}{f}=\lim_{z\rightarrow z_j}\frac{F}{f}=1$.
	Note that
	\begin{equation}
		\nonumber\begin{split}
			\lim_{z\rightarrow z_1}\frac{g\left(\frac{d(\tilde{f}_{z_1}H_3)}{\tilde{f}_{z_1}H_3}+\frac{p_jd\tilde{f}_{z_j}}{p_1\tilde{f}_{z_j}}\right)}{f}	=\lim_{z\rightarrow z_1}\frac{g\frac{d(\tilde{f}_{z_1}H_3)}{\tilde{f}_{z_1}H_3}}{f}
			=\frac{1}{ord_{z_1}g}\lim_{z\rightarrow z_1}\frac{dg}{f}
	\end{split} \end{equation}
	and
	\begin{equation}
		 \nonumber\begin{split}
			\lim_{z\rightarrow z_j}\frac{g\left(\frac{d(\tilde{f}_{z_1}H_3)}{\tilde{f}_{z_1}H_3}+\frac{p_jd\tilde{f}_{z_j}}{p_1\tilde{f}_{z_j}}\right)}{f}	=\lim_{z\rightarrow z_j}\frac{g\frac{p_jd\tilde{f}_{z_j}}{p_1\tilde{f}_{z_j}}}{f}	
			=\frac{p_j}{p_1ord_{z_j}g}\lim_{z\rightarrow z_j}\frac{dg}{f}.
		\end{split}
	\end{equation}	
	Hence,
	\begin{displaymath}
		\frac{p_j}{ord_{z_j}g}\lim_{z\rightarrow z_j}\frac{dg}{f}=\frac{p_1}{ord_{z_1}g}\lim_{z\rightarrow z_1}\frac{dg}{f}
	\end{displaymath}
	for any $1\le j<\gamma$, which implies statement $(3)$ holds. 
	
	\

	\emph{\textbf{Step 4.} $\sum_{1\le j<\gamma}p_j<+\infty$.}

	Note that $ord_{z_j}f=ord_{z_j}g_0-1$ for any $j$ implies $I_0=Z_0'$. Equality \eqref{eq:211009g} shows
	\begin{equation}
		\label{eq:211102e}G(0)=\int_{\Omega}|F|^2e^{-\varphi}c(-\psi)=\left(\int_0^{+\infty}c(s)e^{-s}ds\right)\sum_{1\le j<\gamma}\frac{|d_{1,j}|^2}{p_j|d_{2,j}|^2}2\pi e^{-2u(z_j)}.
	\end{equation}
	Note that
	\begin{equation}\begin{split}
			\label{eq:211102f}|\frac{p_j}{ord_{z_j}g}\lim_{z\rightarrow z_j}\frac{dg}{f}|= |\frac{p_j}{ord_{z_j}g}\lim_{z\rightarrow z_j}\frac{dg}{F}|
			=|\frac{p_j}{ord_{z_j}g_0}\lim_{z\rightarrow z_j}\frac{g_1dg_0}{F}|
			=|p_je^{u(z_j)}\frac{d_{2,j}}{d_{1,j}}|.
	\end{split} \end{equation}
	Combining equality \eqref{eq:211102e}, \eqref{eq:211102f} and $\frac{p_j}{ord_{z_j}g}\lim_{z\rightarrow z_j}\frac{dg}{f}=c_0$, we obtain that
	\begin{equation*}\begin{split}
			G(0)&=\left(\int_0^{+\infty}c(s)e^{-s}ds\right)\sum_{j\in \mathbb{Z}_{\ge1}}\frac{|d_{1,j}|^2}{p_j|d_{2,j}|^2}2\pi e^{-2u(z_j)}\\	&=\left(\int_0^{+\infty}c(s)e^{-s}ds\right)\sum_{1\le j<\gamma}\frac{2\pi p_j}{|c_0|^2},
	\end{split}\end{equation*}
	which implies that $\sum_{1\le j<\gamma}p_j<+\infty$.
\end{proof}

When  $\gamma=m+1$, by a similar discuss of the above Step 3 and following the notations  $\chi_{z_j}$, $\chi_{-u}$, $f_{z_j}$ and $f_u$ in Section \ref{sec:1.2}, we know that 
\begin{Remark}\label{r:equi}For any $1\le j\le m,$ assume $ord_{z_j}f={\frac{1}{2}\nu(dd^c(\varphi+\psi),z_j)}-1$, and let $l_j\in\{0,1,\ldots,ord_{z_j}f\}$.
The following two statements are equivalent:

$(1)$ There exist $g\in\mathcal{O}(\Omega)$ and a constant $c_0\in\mathbb{C}\backslash\{0\}$ such that $\varphi+\psi=2\log|g|$ and
 $\frac{p_j}{ord_{z_j}g}\lim_{z\rightarrow z_j}\frac{dg}{f}=c_0$ for any $j$;

$(2)$ There exist $g\in\mathcal{O}(\Omega)$ and a constant $c_0\in\mathbb{C}\backslash\{0\}$ such that $\varphi+\psi=2\log|g|+2\sum_{1\le j\le m}(l_j+1)G_{\Omega}(\cdot,z_j)+2u$, $\prod_{1\le j\le m}\chi_{z_j}^{l_j+1}=\chi_{-u}$, and for any $k\in\{1,2,\ldots,m\}$,
$\lim_{z\rightarrow z_k}\frac{f}{gP_*\Big(f_u\big(\prod\limits_{1\le j\le m}f_{z_j}^{l_j+1}\big)\big(\sum_{1\le j\le m}p_{j}\frac{d{f_{z_{j}}}}{f_{z_{j}}}\big)\Big)}=c_0.$
\end{Remark}

\subsection{Application: $L^2$ extension from infinite points to $\Omega$}\label{sec:5.3}
\

In this section, we prove that the equality in the optimal $L^2$ extension theorem from infinite points to $\Omega$ (see Proposition \ref{p:extension}) does not hold, i.e. the ``$\le$" in the estimate can be refined to ``$<$".

Let $\Omega$, $\psi$, $\varphi$, $c(t)$, $Z_0'$, $z_j$, $w_j$, $V_{z_j}$ and $V_0$  be as in Proposition \ref{p:extension}. Assume that $\gamma=+\infty$.
Let $c_{\beta}(z)$ be the logarithmic capacity (see \cite{S-O69}) on $\Omega$, which is locally defined by
$$c_{\beta}(z_j):=\exp\lim_{z\rightarrow z_j}(G_{\Omega_j}(z,z_j)-\log|w_j(z)|).$$

Using Proposition \ref{p:extension} and \ref{p:infinite}, we obtain

\begin{Theorem}\label{p:infinite-extension}
	Assume $\frac{1}{2}v(dd^{c}\psi,z_j)=\frac{1}{2}\nu(dd^c(\varphi+\psi),z_j)=k_j+1$ and $\alpha_j:=(\varphi+\psi-2(k_j+1)G_{\Omega}(\cdot,z_j))(z_j)>-\infty$ for any $1\le j<+\infty$, where $k_j\in\mathbb{Z}_{\ge0}$. 
	Let $f\in H^0(V_0,\mathcal{O}(K_{\Omega}))$ satisfy $f=a_jw_j^{k_j}$ on $V_{z_j}$ for any $j$, where
	$a_j$ is a sequence of constants such that 
	$$\sum_{j\in\mathbb{Z}_{\ge1}}\frac{2\pi|a_j|^2e^{-\alpha_j}}{(k_j+1)c_{\beta}(z_j)^{2(k_j+1)}}\in(0,+\infty).$$ Then there exists an $F\in H^0(\Omega,\mathcal{O}(K_{\Omega}))$ such that $F=f+o(w_j^{k_j})dw_j$ near ${z_j}$ for any $j$ and
	\begin{equation}
		\nonumber
		\int_{\Omega}|F|^2e^{-\varphi}c(-\psi)<\left(\int_0^{+\infty}c(s)e^{-s}ds\right)\sum_{j\in\mathbb{Z}_{\ge1}}\frac{2\pi|a_j|^2e^{-\alpha_j}}{(k_j+1)c_{\beta}(z_j)^{2(k_j+1)}}.
	\end{equation}	
\end{Theorem}

\begin{proof}As $c(t)e^{-t}$ is decreasing on $(0,+\infty)$,  following from Lemma \ref{l:green-sup2}  we have $\psi\leq\tilde\psi:=2\sum_{1\le j<\gamma }p_jG_{\Omega}(\cdot,z_j)$ and
	$e^{-\varphi}c(-\psi)\leq e^{-(\varphi+\psi-\tilde\psi)}c(-\tilde\psi).$ Thus, we can assume that $\psi=\tilde\psi=2\sum_{1\le j<\gamma }p_jG_{\Omega}(\cdot,z_j)$.
	
	There exist $g_0\in\mathcal{O}(\Omega)$ and $u_0\in SH(\Omega)$ satisfying $\nu(dd^cu_0,z)\in[0,1)$ for any $z\in\Omega$, such that
$
		\varphi+\psi=2\log|g_0|+2u_0.
$
 Note that $ord_{z_j}g_0=k_j+1$  and
	$e^{2u_0(z_j)}\lim_{z\rightarrow z_j}\left|\frac{g_0}{w_j^{k_j+1}}(z)\right|^2=e^{\alpha_j}c_{\beta}(z_j)^{2(k_j+1)}.$
	 Proposition \ref{p:extension} shows  that there exists an $F_0\in H^0(\Omega,\mathcal{O}(K_{\Omega}))$ such that $F_0=f+o(w_j^{k_j})dw_j$ near ${z_j}$ for any $j$ and
	\begin{equation*}
		\int_{\Omega}|F_0|^2e^{-\varphi}c(-\psi)\le\left(\int_0^{+\infty}c(s)e^{-s}ds\right)\sum_{j\in\mathbb{Z}_{\ge1}}\frac{2\pi|a_j|^2e^{-\alpha_j}}{(k_j+1)c_{\beta}(z_j)^{2(k_j+1)}}.
	\end{equation*}
	
	Denote the minimal $L^2$ integral of holomorphic extensions by $C_{\Omega,f}.$
	Now, we assume $\left(\int_0^{+\infty}c(s)e^{-s}ds\right)\sum_{j\in\mathbb{Z}_{\ge1}}\frac{2\pi|a_j|^2e^{-\alpha_j}}{(k_j+1)c_{\beta}(z_j)^{2(k_j+1)}}=C_{\Omega,f}$ to get a contradiction.
	
	Similarly, for any $t>0$, Proposition \ref{p:extension} ($\psi\sim\psi+t$, $\varphi\sim\varphi-t$ $c(\cdot)\sim c(\cdot+t)e^{-t}$ and $\Omega\sim\{\psi<-t\}$)  shows that there exists an  $F_t\in H^0(\{\psi<-t\},\mathcal{O}(K_{\Omega}))$  such that $F_t=f+o(w_j^{k_j})dw_j$ near ${z_j}$ for any $j$ and
	$$\int_{\{\psi<-t\}}|F_t|^2e^{-\varphi}c(-\psi)\leq\left(\int_t^{+\infty}c(s)e^{-s}ds\right)\sum_{j\in\mathbb{Z}_{\ge1}}\frac{2\pi|a_j|^2e^{-\alpha_j}}{(k_j+1)c_{\beta}(z_j)^{2(k_j+1)}}.$$
	As $e^{-\varphi}c(-\psi)=e^{-\varphi-\psi}e^{\psi}c(-\psi)$ and $c(t)e^{-t}$ is decreasing on $(0,+\infty)$, $e^{-\varphi}c(-\psi)$ has locally positive lower bound on $\Omega\backslash Z_0'$. Taking $\mathcal{F}_{z_j}=\mathcal{I}(2(k_j+1)G_{\Omega}(\cdot,z_j))_{z_j}$, by the definition of $G(t)$,  we obtain that inequality
	\begin{equation}
		\label{eq:211108b}
		\frac{G(t)}{\int_t^{+\infty}c(s)e^{-s}ds}\leq	\sum_{j\in\mathbb{Z}_{\ge1}}\frac{2\pi|a_j|^2e^{-\alpha_j}}{(k_j+1)c_{\beta}(z_j)^{2(k_j+1)}}=\frac{G(0)}{\int_0^{+\infty}c(s)e^{-s}ds}\end{equation}
	holds for any $t\geq0$.
	Following from equality \eqref{eq:211108b} and Theorem \ref{maintheorem}, we have that $G( {h}^{-1}(r))$ is linear with respect to $r$. Note that $\sum_{j\in\mathbb{Z}_{\ge1}}k_j+1=+\infty$, which contradicts to Proposition \ref{p:infinite}. Thus, Theorem \ref{p:infinite-extension} holds.
\end{proof}

\section{Proofs of Theorem \ref{thm:m-points} and Remark \ref{rem:1.1}}
In this section, we prove Theorem \ref{thm:m-points} and Remark \ref{rem:1.1}.

The  necessity in Theorem \ref{thm:m-points} holds by Proposition \ref{p:infinite} and Remark \ref{r:equi}, then it suffices to prove the sufficiency.
By Remark \ref{r:c'}, assume case $p_j>2$ for any $j$ without loss of generality.

 Let $F=c_0gP_*\left(f_u(\prod_{1\le j\le m}f_{z_j})\left(\sum_{1\le j\le m}p_{j}\frac{d{f_{z_{j}}}}{f_{z_{j}}}\right)\right)$  on $\Omega$, which is a (single-value) holomorphic $(1,0)$ form on $\Omega$ by $\prod_{1\le j\le m}\chi_{z_j}=\chi_{-u}$, where $g$ is the holomorphic function in statement $(2)$ and $c_0$ is the constant in statement $(4)$.
As $\varphi+\psi=2\log|g|+2\sum_{1\le j\le m}G_{\Omega}(\cdot,z_j)+2u$, $\mathcal{F}_{z_j}=\mathcal{I}(\varphi+\psi)_{z_j}$, $ord_{z_j}(g)=ord_{z_j}(f)$ and $\lim_{z\rightarrow z_j}\frac{f}{F}=1$   for any $j$,  we have $(F-f,z_j)\in\mathcal{O}((K_{\Omega}))_{z_j}\otimes\mathcal{F}_{z_j}$ for any $j$.
Note that $|P_*(f_u)|=e^u$ and $|P_*f_{z_j}|=e^{G_{\Omega}(\cdot,z_j)}$. Then we have
\begin{equation}
\nonumber
	\begin{split}
		|F|^2e^{-\varphi}&=|c_0|^2\Bigg|P_*\bigg(\big(\prod_{1\le j\le m}f_{z_j}\big)\bigg(\sum_{1\le j\le m}p_{j}\frac{d{f_{z_{j}}}}{f_{z_{j}}}\bigg)\bigg)\Bigg|^2e^{2\sum_{1\le j\le m}(p_j-1)G_{\Omega}(\cdot,z_j)}\\
		&=|c_0|^2P_*\bigg(\prod_{1\le j\le m}|f_{z_j}|^{2p_j}\bigg)\Bigg|P_*\bigg(\sum_{1\le j\le m}p_j\frac{df_{z_j}}{f_{z_j}}\bigg)\Bigg|^2.
	\end{split}
\end{equation}
Note that $p_{j}>2$.  
Combining with equality  \eqref{eq:211026h}, we obtain that
$
	|F|^2e^{-\varphi}=\sqrt{-1}|c_0|^2\partial\overline{\partial}e^{2\sum_{1\le j\le m}p_jG_{\Omega}(\cdot,z_j)}.
$
Using Lemma \ref{l:4}, we get
\begin{equation}
	\label{eq:211021a}\int_{\Omega}|F|^2e^{-\varphi}=\sqrt{-1}|c_0|^2\int_{\Omega}\partial\overline{\partial}e^{2\sum_{1\le j\le m}p_jG_{\Omega}(\cdot,z_j)}=2\pi|c_0|^2\sum_{1\le j\le m}p_j.
\end{equation}

For any $\tilde{F}\in H^0(\Omega,\mathcal{O}(K_{\Omega}))$ satisfying $(\tilde{F}-f,z_j)\in\mathcal{O}((K_{\Omega}))_{z_j}\otimes\mathcal{F}_{z_j}$ for any $j$ and $\int_{\Omega}|\tilde{F}|^2e^{-\varphi}<+\infty$, there exists a  $\beta\in H^0(\Omega,\mathcal{O}(K_{\Omega}))$  such that  $\frac{\tilde{F}-F}{g}=P_*(f_u\prod_{1\le j\le m}f_{z_j})\beta$ and
\begin{displaymath}\begin{split}
	\int_{\Omega}|P_*(\prod_{1\le j\le m}f_{z_j})\beta|^2e^{-2\sum_{1\le j\le m}(p_j-1)G_{\Omega}(\cdot,z_j)}=\int_{\Omega}|\tilde{F}-F|^2e^{-\varphi}<+\infty,
	\end{split}
\end{displaymath}
which implies that $\int_{\Omega}|\beta|^2<+\infty$.
Note that
\begin{equation}
\nonumber
	\begin{split}
		&F\wedge\overline{(F-\tilde{F})}e^{-\varphi}\\
		=&c_0gP_*\Bigg(f_u\bigg(\prod_{1\le j\le m}f_{z_j}\bigg)\bigg(\sum_{1\le j\le m}p_{j}\frac{d{f_{z_{j}}}}{f_{z_{j}}}\bigg)\Bigg)\wedge\overline{gP_*\bigg(f_u\prod_{1\le j\le m}f_{z_j}\bigg)\beta}e^{-\varphi}\\
		=&c_0e^{2\sum_{1\le j\le m}p_jG_{\Omega}(\cdot,z_j)}P_*\bigg(\sum_{1\le j\le m}p_j\frac{df_{z_j}}{f_{z_j}}\bigg)\wedge\overline\beta.
	\end{split}
\end{equation}
and
\begin{equation}
\nonumber
	\begin{split}
		\partial e^{2\sum_{1\le j\le m}p_jG_{\Omega}(\cdot,z_j)}&=\sum_{1\le k\le m}e^{2\sum_{1\le j\le m}p_jG_{\Omega}(\cdot,z_j)-2G_{\Omega}(\cdot,z_k)}p_kP_*(\overline{f_{z_k}}\partial f_{z_k})\\
		&=e^{2\sum_{1\le j\le m}p_jG_{\Omega}(\cdot,z_j)}P_*\bigg(\sum_{1\le j\le m}p_j\frac{df_{z_k}}{f_{z_k}}\bigg).
	\end{split}
\end{equation}
Then we have
$
	F\wedge\overline{(F-\tilde{F})}e^{-\varphi}=c_0\partial e^{2\sum_{1\le j\le m}p_jG_{\Omega}(\cdot,z_j)}\wedge\overline\beta.
$
Following from Lemma \ref{l:5}, we have 
$\int_{\Omega}F\wedge\overline{(F-\tilde{F})}e^{-\varphi}=\int_{\Omega}c_0\partial e^{2\sum_{1\le j\le m}p_jG_{\Omega}(\cdot,z_j)}\wedge\overline\beta=0,$
which implies that
$\int_{\Omega}|\tilde{F}|^2e^{-\varphi}=\int_{\Omega}|F|^2e^{-\varphi}+\int_{\Omega}|\tilde{F}-F|^2e^{-\varphi}.$
Thus, we have
\begin{equation}
	\label{eq:211021b}
	G(0;\tilde{c}\equiv1)=\int_{\Omega}|F|^2e^{-\varphi}.
\end{equation}
It follows from  equality \eqref{eq:211021a} and  \eqref{eq:211021b} that
\begin{equation}
	\label{eq:211021c}
	G(0;\tilde{c}\equiv1)=2\pi|c_0|^2\sum_{1\le j\le m}p_j.
\end{equation}

Let $\tilde{w}_j$ be a local coordinate on a neighborhood $\tilde{V}_{z_j}$ satisfying $|\tilde{w}_j|=e^{\sum_{1\le j\le m}G_{\Omega}(\cdot,z_j)}=|P_*(\prod_{1\le j\le m}f_{z_j})|$ on $\tilde{V}_{z_j}$ for any $j\in\{1,2,\ldots,m\}$. Note that $\varphi+\psi=2\log|g\tilde{w}_j|+2u$ on $\tilde{V}_{z_j}$.
Assume that $f=d_{1,j}\tilde{w}_j^{k_{1,j}}h_{1,j}d\tilde{w}_j$ and $g\tilde{w}_j=d_{2,j}\tilde{w}_j^{k_{2,j}}h_{2,j}$, where $d_{i,j}\not=0$ is constant, $k_{i,j}$  is nonnegative integer, and $h_{i,j}\in\mathcal{O}(\tilde{V}_{z_j})$  satisfying $h_{i,j}(z_j)=1$  for any $i\in\{1,2\}$ and $j\in\{1,2,\ldots,m\}$. Note that $k_{1,j}+1=k_{2,j}$ and
\begin{equation}
\nonumber |c_0|^2=\lim_{z\rightarrow z_j}\frac{f}{gP_*\left(f_u(\prod_{1\le j\le m}f_{z_j})\left(\sum_{1\le j\le m}p_{j}\frac{d{f_{z_{j}}}}{f_{z_{j}}}\right)\right)}=\frac{|d_{1,j}|^2e^{-2u(z_j)}}{|p_jd_{2,j}|^2}
\end{equation}
 for any $j\in\{1,2,\ldots,m\}$.
Proposition \ref{p:extension} shows that
\begin{equation}
	\label{eq:211021d}G(t;\tilde{c}\equiv1)\le e^{-t}\sum_{1\le j\le m}\frac{2\pi|d_{1,j}|^2e^{-2u(z_j)}}{p_j|d_{2,j}|^2}=2\pi|c_0|^2e^{-t}\sum_{1\le j\le m}p_j.
\end{equation}
As  $G(-\log r;\tilde{c}\equiv1)$ is concave with respect to $r$ by Theorem \ref{maintheorem}, using equality \eqref{eq:211021c} and inequality \eqref{eq:211021d}, we know $G(-\log r,\tilde{c}\equiv1)$ is linear with respect to $r$.
It follows from Proposition \ref{c:linear} and Remark \ref{r:c} that $G(h^{-1}(r);c)$ is linear with respect to $r$.

Thus,  Theorem \ref{thm:m-points} holds.

Now, we prove Remark \ref{rem:1.1}. 
As $G(-\log r;\tilde{c})$ is linear with respect to $r$, it follows from equality \eqref{eq:211021b} Proposition \ref{c:linear} and Lemma \ref{l:unique} shows that
$$
G(t;\tilde{c}\equiv1)=\int_{\{\psi<-t\}}|F|^2e^{-\varphi}
$$ for any $t\ge0$, where $F=c_0gP_*(f_u(\prod_{1\le j\le m}f_{z_j})(\sum_{1\le j\le m}p_{j}\frac{d{f_{z_{j}}}}{f_{z_{j}}}))\in H^0(\Omega,\mathcal{O}(K_{\Omega}))$.
By Proposition \ref{c:linear} and Remark \ref{r:c},
$$G(t;c)=\int_{\{\psi<-t\}}|F|^2e^{-\varphi}c(-\psi)$$
for any $t\ge0.$ The uniqueness follows from Proposition \ref{c:linear}.
Thus, Remark \ref{rem:1.1} holds.

 \section{Proofs of Theorem \ref{c:L2-1d-char} and Remark \ref{rem:1.2}}

 In this section, we prove Theorem \ref{c:L2-1d-char} and Remark \ref{rem:1.2}.

Using the Weierstrass Theorem on open Riemann surface (see \cite{OF81}) and  Siu's Decomposition Theorem \cite{siu74}, we have
\begin{equation}
\nonumber
	\varphi+\psi=2\log|g_0|+2u_0,
\end{equation}
where $g_0\in\mathcal{O}(\Omega)$ and $u_0\in SH(\Omega)$ such that $\nu(dd^cu_0,z)\in[0,1)$ for any $z\in\Omega$. Note that $ord_{z_j}g_0=k_j+1$  and
$e^{2u_0(z_j)}\lim_{z\rightarrow z_j}|\frac{g_0}{w_j^{k_j+1}}(z)|^2=e^{\alpha_j}c_{\beta}(z_j)^{2(k_j+1)}.$
By Proposition \ref{p:extension},  there exists a minimal extension form $F\in H^0(\Omega,\mathcal{O}(K_{\Omega}))$ such that $F=f+o(w_j^{k_j})dw_j$ near $z_j$ for any $j\in\{1,2,\ldots,m\}$ and
\begin{equation*}
	\int_{\Omega}|F|^2e^{-\varphi}c(-\psi)\le\left(\int_0^{+\infty}c(s)e^{-s}ds\right)\sum_{1\le j\le m}\frac{2\pi|a_j|^2e^{-\alpha_j}}{p_jc_{\beta}(z_j)^{2(k_j+1)}}.
\end{equation*}

In the following, we prove the characterization of the holding of the equality $C_{\Omega,f}=\left(\int_0^{+\infty}c(s)e^{-s}ds\right)\sum_{1\le j\le m}\frac{2\pi|a_j|^2e^{-\alpha_j}}{p_jc_{\beta}(z_j)^{2(k_j+1)}}$.

Similarly, for any $t>0$, Proposition \ref{p:extension} ($\psi\sim\psi+t$, $\varphi\sim\varphi-t$ $c(\cdot)\sim c(\cdot+t)e^{-t}$ and $\Omega\sim\{\psi<-t\}$)  shows that there exists an  $F_t\in H^0(\{\psi<-t\},\mathcal{O}(K_{\Omega}))$  such that $F_t=f+o(w_j^{k_j})dw_j$ near ${z_j}$ for any $j$ and
$$\int_{\{\psi<-t\}}|F_t|^2e^{-\varphi}c(-\psi)\leq\left(\int_t^{+\infty}c(s)e^{-s}ds\right)\sum_{1\le j\le m}\frac{2\pi|a_j|^2e^{-\alpha_j}}{p_jc_{\beta}(z_j)^{2(k_j+1)}}.$$

Firstly, we prove the necessity. Assume that equality \eqref{eq:241208b} holds.
 Take $\mathcal{F}_{z_j}=\mathcal{I}(2(k_j+1)G_{\Omega}(\cdot,z_j))_{z_j}$ for any $j$, and denote
$$\inf\bigg\{\int_{\{\psi<-t\}}|\tilde f|^2e^{-\varphi}c(-\psi):\tilde f\in H^0(\{\psi<-t\},\mathcal{O}(K_{\Omega}))\,\&\,(\tilde f-f,z_j)\in\mathcal{F}_{z_j}\,\forall j\bigg\}$$
by $G(t)$, where $t\ge 0$.
Then we have
\begin{equation}
	\label{eq:241207a}
	\frac{G(t)}{\int_t^{+\infty}c(s)e^{-s}ds}\leq	\sum_{1\le j\le m}\frac{2\pi|a_j|^2e^{-\alpha_j}}{p_jc_{\beta}(z_j)^{2(k_j+1)}}=\frac{G(0)}{\int_0^{+\infty}c(s)e^{-s}ds}.\end{equation}
Denote $\tilde\psi:=2\sum_{1\le j\le m}p_jG_{\Omega}(\cdot,z_j)$. 
By the above discussion, there exists a $\tilde{F}_1\in H^0(\Omega,\mathcal{O}(K_{\Omega}))$  such that $\tilde{F}_1=f+o(w_j^{k_j})dw_j$ near ${z_j}$ for any $j$ and
\begin{equation}\label{eq:241208c}
\int_{\Omega}|\tilde{F}_1|^2e^{-(\varphi+\psi-\tilde\psi)}c(-\tilde\psi)
		\leq\left(\int_0^{+\infty}c(s)e^{-s}ds\right)\sum_{1\le j\le m}\frac{2\pi|a_j|^2e^{-\alpha_j}}{p_jc_{\beta}(z_j)^{2(k_j+1)}}.
\end{equation}
As $c(t)e^{-t}$ is decreasing and $\psi\leq\tilde\psi$,  $e^{-\varphi}c(-\psi)\le e^{-\varphi-\psi+\tilde\psi}c(-\tilde\psi)$. Thus,
it follows from equality \eqref{eq:241208b} and \eqref{eq:241208c} that
$\int_{\Omega}|\tilde{F}_1|^2e^{-\varphi}c(-\psi)=\int_{\Omega}|\tilde{F}_1|^2e^{-(\varphi+\psi-\tilde\psi)}c(-\tilde\psi).$
Using Lemma \ref{l:psi=G}, we have  $\psi=\tilde\psi=2\sum_{1\le j\le m}p_jG_{\Omega}(\cdot,z_j)$.
 Using the concavity of $G({h}^{-1}(r))$ (see Theorem \ref{maintheorem}) and inequality \eqref{eq:241207a}, we know that $G({h}^{-1}(r))$ is linear on $(0,\int_0^{+\infty}c(t)e^{-t}dt)$, where $h(t)=\int_t^{+\infty}c(s)e^{-s}ds$. Hence by Theorem \ref{thm:m-points} and remark \ref{r:equi}, statements $(2)$, $(3)$ and $(4)$ in Theorem \ref{c:L2-1d-char} hold.

Now, we prove the sufficiency.
Assume that the four statements in Theorem \ref{c:L2-1d-char} hold. 
It follows from Theorem \ref{thm:m-points} and Remark \ref{r:equi}  that $G( {h}^{-1}(r))$ is linear with respect to $r$.  Proposition \ref{c:linear} shows that the minimal extension forms on all sublevel sets are the same, i.e., 
$$\int_{\{\psi<-t\}}|{F}|^2e^{-\varphi}c(\psi)=G(t)$$
for any $t$.
Let $\tilde{w}_j$ be a local coordinate on a neighborhood $\tilde{V}_{z_j}\subset V_{z_j}$ of $z_j$ satisfying $\log|\tilde{w}_j|=\frac{1}{p_j}\sum_{1\le k\le m}p_kG_{\Omega}(\cdot,z_k)$. As $f=a_jw_j^{k_j}dw_j$ on $V_{z_j}$,  we have
\begin{equation}
	\label{eq:211028a}
	\lim_{z\rightarrow z_j}\left|\frac{{F}(z)}{\tilde{w}_j(z)^{k_j}d\tilde{w}_j}\right|^2=|a_j|^2\left(\lim_{z\rightarrow z_j}\left|\frac{w_j(z)}{\tilde{w}_j(z)}\right|\right)^{2(k_j+1)}
\end{equation}
and
\begin{equation}
	\label{eq:211028b}\begin{split}
	\lim_{z\rightarrow z_j}e^{-\varphi(z)-\psi(z)}|\tilde{w}_j(z)|^{2(k_j+1)}&=\lim_{z\rightarrow z_j}e^{-\varphi(z)-\psi(z)+2(k_j+1)G_{\Omega}(z,z_j)}\frac{|\tilde{w}_j(z)|^{2(k_j+1)}	}{e^{2(k_j+1)G_{\Omega}(z,z_j)}}	\\
	&=e^{-\alpha_j}\big(\exp\lim_{z\rightarrow z_j}(\log|\tilde{w}_j|-G_{\Omega}(z,z_j))\big)^{2(k_j+1)}.
	\end{split}\end{equation}
Combining equality \eqref{eq:211028a}, equality \eqref{eq:211028b} and $c_{\beta}(z_j)=\exp\lim_{z\rightarrow z_j}(G_{\Omega}(z,z_j)-\log|w(z)|)$, we have
\begin{equation}
	\label{eq:211028c}
	\lim_{z\rightarrow z_j}\left|\frac{{F}(z)}{d\tilde{w}_j}\right|^2e^{-\varphi(z)-\psi(z)}|\tilde{w}_j(z)|^2=\frac{|a_j|^2e^{-\alpha_j}}{c_{\beta}(z_j)^{2(k_j+1)}}.
\end{equation}
Using Lemma \ref{l:G-compact}, there exists a $t_0>0$ such that $\{2\sum_{1\le j\le m}p_jG_{\Omega}(\cdot,z_j)<-t_0\}\Subset \cup_{1\le j\le m}\tilde{V}_{z_j}$. It follows from equality \eqref{eq:211028c} that
\begin{displaymath}
	\begin{split}
		&\lim_{t\rightarrow+\infty}\frac{\int_{\{\psi<-t\}}|{F}|^2e^{-\varphi}c(-\psi)}{\int_{t}^{+\infty}c(s)e^{-s}ds}\\
		=&\lim_{t\rightarrow+\infty}\sum_{1\le j\le m}\frac{\int_{\{2p_j\log|\tilde{w}_j|<-t\}}|\frac{{F}}{d\tilde{w}_j}|^2e^{-\varphi-\psi}|\tilde{w}_j|^{2p_j}c(-2p_j\log|\tilde{w}_j|)\sqrt{-1}d\tilde{w}_j\wedge d\overline{\tilde{w}_j}}{\int_{t}^{+\infty}c(s)e^{-s}ds}\\
		=&\sum_{1\le j\le m}\frac{|a_j|^2e^{-\alpha_j}}{c_{\beta}(z_j)^{2(k_j+1)}}\lim_{t\rightarrow+\infty}\frac{\int_{\{2p_j\log|\tilde{w}_j|<-t\}}|\tilde{w}_j|^{2(p_j-1)}c(-2p_j\log|\tilde{w}_j|)\sqrt{-1}d\tilde{w}_j\wedge d\overline{\tilde{w}_j}}{\int_{t}^{+\infty}c(s)e^{-s}ds}\\
		=&\sum_{1\le j\le m}\frac{4\pi|a_j|^2e^{-\alpha_j}}{c_{\beta}(z_j)^{2(k_j+1)}}\lim_{t\rightarrow+\infty}\frac{\int_{0}^{e^{-\frac{t}{2p_j}}}r^{2(p_j-1)+1}c(-2p_j\log r)dr}{\int_{t}^{+\infty}c(s)e^{-s}ds}  \\
		=&\sum_{1\le j\le m}\frac{2\pi|a_j|^2e^{-\alpha_j}}{p_jc_{\beta}(z_j)^{2(k_j+1)}},
	\end{split}
\end{displaymath}
Thus the equality $\left(\int_0^{+\infty}c(s)e^{-s}ds\right)\sum_{1\le j\le m}\frac{2\pi|a_j|^2e^{-\alpha_j}}{p_jc_{\beta}(z_j)^{2(k_j+1)}}=G(0)=C_{\Omega,f}$ holds.

Thus, Theorem \ref{c:L2-1d-char} holds.

Finally, we prove Remark \ref{rem:1.2}.
 As $G(h^{-1}(r))$ is linear, it follows from Remark \ref{rem:1.1} that the minimal extension form
$$F=cg_1P_*\Bigg(f_{u_1}\bigg(\prod_{1\le j\le m}f_{z_j}\bigg)\bigg(\sum_{1\le j\le m}p_{j}\frac{d{f_{z_{j}}}}{f_{z_{j}}}\bigg)\Bigg),$$
which is  the unique holomorphic $(1,0)$ form on $\Omega$ such that $F=f+o(w_j^{k_j})dw_j$ near $z_j$ for any $j$ and
$$\int_{\Omega}|F|^2e^{-\varphi}c(-\psi)\leq\left(\int_0^{+\infty}c(s)e^{-s}ds\right)\sum_{1\le j\le m}\frac{2\pi|a_j|^2e^{-\alpha_j}}{p_jc_{\beta}(z_j)^{2(k_j+1)}}.$$ Here $g_1\in\mathcal{O}(\Omega)$ and $u_1$ is harmonic such that $ord_{z_j}g_1=k_j$ and $\varphi+\psi=2\log|g_1|+2\sum_{1\le j\le m}G_{\Omega}(\cdot,z_j)+2u_1$.
As $\varphi+\psi=2\log|g|+2\sum_{1\le j\le m}(k_j+1)G_{\Omega}(\cdot,z_j)+2u$, we have
 $f_u=f_{u_1}\frac{P^*(g_2)}{\prod_{1\le j\le m}f_{z_j}^{k_j}}$ on $\Delta$, where $g_2=\frac{g_1}{g}\in\mathcal{O}(\Omega)$. Then we obtain that
$$F=c_0gP_*\Bigg(f_{u}\bigg(\prod_{1\le j\le m}f_{z_j}^{k_j+1}\bigg)\bigg(\sum_{1\le j\le m}p_{j}\frac{d{f_{z_{j}}}}{f_{z_{j}}}\bigg)\Bigg).$$
 
\section{Appendix: An example of Theorem \ref{c:L2-1d-char}}
\label{appendix}

Let $\Delta$ be the unit disc in $\mathbb{C}$, and let $Z_0=\{z_1=0,z_2=\frac{1}{2}\}\subset\Delta$. Let $k_1=1$ and $k_2=0$. Note that $G_{\Delta}(z,0)=\log|z|$ and $G_{\Delta}(z,\frac{1}{2})=\log|\frac{2z-1}{2-z}|$. Let $\psi=4\log|z|+2\log|\frac{2z-1}{2-z}|$, and let $\varphi=0$. Then we have $c_{\beta}(z_1)=1$, $c_{\beta}(z_2)=\frac{4}{3}$, $\alpha_1=-\log 4$ and $\alpha_2=-4\log 2$. Let $f$ be a holomorphic $(1,0)$ form on $\{|z|<\frac{1}{10}\}\cup\{|\frac{2z-1}{2-z}|<\frac{1}{10}\}$ such that $f=zdz$ on $\{|z|<\frac{1}{10}\}$ and $f=adz$ on $\{|\frac{2z-1}{2-z}|<\frac{1}{10}\}$, where $a\in\mathbb{C}\backslash\{0\}$, and let $c\equiv1$ be a function on $(0,+\infty)$. Then we have the right hand side in inequality \eqref{eq:210902a} is $4\pi+18|a|^2\pi$.

 For any  $F\in H^0(\Delta,\mathcal{O}(K_{\Omega}))$, there exists $\{a_l\}_{l\in\mathbb{Z}_{\ge0}}$ such that $F=\sum_{l\in\mathbb{Z}_{\ge0}}a_lz^ldz$. $F$ satisfies $(F-f,z_j)\in(\mathcal{O}(K_{\Delta})\otimes\mathcal{I}(\varphi+\psi))_{z_j}$ for any $j\in\{1,2\}$  if and only if $a_0=0$, $a_1=1$ and $\sum_{l\in\mathbb{Z}_{\ge0}}a_l(\frac{1}{2})^{l}=a$. By a direct calculation, we have
 \begin{equation}
 	\label{eq:211028d}\begin{split}
 		\int_{\Delta}|F|^2&=\lim_{r\rightarrow1-0}\int_{\{|z|<r\}}\sqrt{-1}F\wedge\overline{F}=\lim_{r\rightarrow1-0}4\pi\sum_{l\in\mathbb{Z}_{\ge0}}\frac{|a_l|^2r^{2l+2}}{2l+2}\\
 		&=2\pi\sum_{l\in\mathbb{Z}_{\ge0}}\frac{|a_l|^2}{j+1}=2\pi\sum_{l\in\mathbb{Z}_{>0}}\frac{|a_l|^2}{j+1}.
 	\end{split}
 \end{equation}
Assume that $(F-f,z_j)\in(\mathcal{O}(K_{\Delta})\otimes\mathcal{I}(2(k_j+1)G_{\Delta}(\cdot,z_j)))_{z_j}$ for any $j\in\{1,2\}$. It follows from Cauchy-Schwarz inequality and equality \eqref{eq:211028d} that
\begin{equation*}
	\begin{split}
		|a-\frac{1}{2}|^2&=|\sum_{l\in\mathbb{Z}_{>1}}a_l(\frac{1}{2})^{l}|^2=|\sum_{l\in\mathbb{Z}_{>1}}\frac{a_l}{\sqrt{l+1}}\sqrt{l+1}(\frac{1}{2})^{l}|^2\\
		&\le(\sum_{l\in\mathbb{Z}_{>1}}\frac{|a_l|^2}{l+1})(\sum_{l\in\mathbb{Z}_{>1}}(l+1)\frac{1}{4^l})=\frac{5}{36\pi}(\int_{\Delta}|F|^2-\pi),
	\end{split}
\end{equation*}which implies that
\begin{equation}
	\label{eq:211028e}\begin{split}
		\int_{\Delta}|F|^2&\ge\frac{36\pi}{5}|a-\frac{1}{2}|^2+\pi.
	\end{split}
\end{equation}
Note that equality \eqref{eq:211028e} becomes equality if and only if $\frac{2^la_l}{l+1}=\frac{2^{l'}a_{l'}}{l'+1}$ for any $l,l'\in\mathbb{Z}_{>1}$. Since there exists a sequence of complex numbers $\{a_l\}_{l\in\mathbb{Z}_{\ge0}}$ satisfying $a_0=0$, $a_1=1$, $\frac{2^la_l}{l+1}=\frac{2^{l'}a_{l'}}{l'+1}$ for any $l,l'\in\mathbb{Z}_{>1}$ and $\sum_{l\in\mathbb{Z}_{\ge0}}a_l(\frac{1}{2})^l=a$, then we obtain that there exists an $F\in H^0(\Delta,\mathcal{O}(K_{\Delta}))$ such that  $\int_{\Delta}|F|^2=\inf\{\int_{\Delta}|\tilde{F}|^2:\tilde{F}\in H^0(\Delta,\mathcal{O}(K_{\Delta}))$ such that $(\tilde{F}-f,z_j)\in(\mathcal{O}(K_{\Delta})\otimes\mathcal{I}(2(k_j+1)G_{\Delta}(\cdot,z_j)))_{z_j}$ for any $j\}=\frac{36\pi}{5}|a-\frac{1}{2}|^2+\pi$.

Following from the right hand side in inequality \eqref{eq:210902a} is $4\pi+18|a|^2\pi$ and
\begin{displaymath}
	\begin{split}
		(4\pi+18|a|^2\pi)-\left(\frac{36\pi}{5}|a-\frac{1}{2}|^2+\pi\right)&=\frac{3\pi}{5}(30|a|^2-12|a-\frac{1}{2}|^2+5)\\
		&=\frac{6\pi}{5}|3a+1|^2\ge0,
	\end{split}
\end{displaymath}
 the inequality \eqref{eq:210902a} holds. Moreover, equality $4\pi+18|a|^2\pi=\inf\{\int_{\Delta}|\tilde{F}|^2:\tilde{F}\in H^0(\Delta,\mathcal{O}(K_{\Delta}))$ such that $(\tilde{F}-f,z_j)\in(\mathcal{O}(K_{\Delta})\otimes\mathcal{I}(2(k_j+1)G_{\Delta}(\cdot,z_j)))_{z_j}$ for any $j\}$ holds if and only if $a=-\frac{1}{3}$.

As $|z|=e^{G_{\Delta}(z,z_1)}$ and $|\frac{2z-1}{2-z}|=e^{G_{\Delta}(z,z_2)}$, then there exists a constant $c_1$ satisfying $|c_1|=1$ and $P_*(f_{z_1}^2f_{z_2})=c_1z^2\frac{2z-1}{2-z}$. Note that
\begin{equation}
	\label{eq:211028f}\lim_{z\rightarrow z_1}\frac{f}{P_*(2f_{z_1}f_{z_2}df_{z_1})}=\lim_{z\rightarrow0}\frac{zdz}{2c_1z\frac{2z-1}{2-z}dz}=-\frac{1}{c_1}
\end{equation}
and \begin{equation}
	\label{eq:211028g}\lim_{z\rightarrow z_2}\frac{f}{P_*(f_{z_1}^2df_{z_1})}=\lim_{z\rightarrow\frac{1}{2}}\frac{adz}{c_1z^2d\frac{2z-1}{2-z}}=\frac{3a}{c_1}.
\end{equation}
It is clear that  statement $(1)$, $(2)$ and $(3)$ in Theorem \ref{c:L2-1d-char} hold. It follows from equality \eqref{eq:211028f} and equality \eqref{eq:211028g} that statement $(4)$ in Theorem \ref{c:L2-1d-char} holds if and only if $a=-\frac{1}{3}$. Thus, equality $4\pi+18|a|^2\pi=\inf\{\int_{\Delta}|\tilde{F}|^2:\tilde{F}\in H^0(\Delta,\mathcal{O}(K_{\Delta}))$ such that $(\tilde{F}-f,z_j)\in(\mathcal{O}(K_{\Delta})\otimes\mathcal{I}(2(k_j+1)G_{\Delta}(\cdot,z_j)))_{z_j}$ for any $j\}$ holds if and only if statement $(4)$ in Theorem \ref{c:L2-1d-char} holds.

\


\vspace{.1in} {\em Acknowledgements}.
The authors would like to thank Dr. Shijie Bao and Dr. Zhitong Mi for checking the manuscript.  The first named author was supported by National Key R\&D Program of China 2021YFA1003100, NSFC-11825101 and NSFC-12425101. The second author was supported by China Postdoctoral Science Foundation BX20230402 and 2023M743719.

\bibliographystyle{references}
\bibliography{xbib}

\end{document}